\documentclass[a4paper,12pt]{amsart}
\usepackage[latin1]{inputenc}
\usepackage[T1]{fontenc}
\usepackage{amsmath}
\usepackage{amssymb}
\usepackage{amsthm}

\usepackage[top=20mm, bottom=20mm, left=25mm, right=25mm]{geometry}

\usepackage{appendix} % appendices
\usepackage{enumitem} % lists
\usepackage{fnpct} % multiple footnotes
\usepackage{graphicx} % pictures
\usepackage{hyperref} % reference

\theoremstyle{theorem}

\newtheorem{theorem}{Theorem}[section]
\newtheorem{proposition}{Proposition}[section]
\newtheorem{lemma}{Lemma}[section]
\newtheorem{claim}{Claim}[section]
\newtheorem{corollary}{Corollary}[section]
\theoremstyle{definition}

\newtheorem{remark}{Remark}[section]

\newcommand{\Q}{\ensuremath{\mathbb{Q}}}
\newcommand{\R}{\ensuremath{\mathbb{R}}}
\newcommand{\Z}{\ensuremath{\mathbb{Z}}}
\newcommand{\N}{\ensuremath{\mathbb{N}}}
\newcommand{\M}{\mathrm{M}}
\newcommand{\GL}{\mathrm{GL}}
\newcommand{\id}{\mathrm{id}}
\newcommand{\rk}{\mathrm{rk}}
\newcommand{\com}{\mathrm{com}}
\newcommand{\transp}{\,{}^t\!}
\newcommand{\abs}[1]{\left\vert #1 \right\vert}
\newcommand{\norme}[1]{\left\Vert #1 \right\Vert}
\renewcommand{\span}{\mathrm{Span}}
\newcommand{\muexp}[4]{\mathring\mu_{#1}(#2\vert #3)_{#4}}
\newcommand{\muexpA}[4]{\mu_{#1}(#2\vert #3)_{#4}}
\renewcommand{\epsilon}{\varepsilon}
\renewcommand{\le}{\leqslant}
\renewcommand{\ge}{\geqslant}

\title{Upper bounds and spectrum for approximation exponents for subspaces of $\R^n$\hspace{3mm}}
\date{\today}
\author[E. Joseph]{Elio Joseph}
\address{Universit\'e Paris-Saclay, CNRS, Laboratoire de math\'ematiques d'Orsay, 91405, Orsay, France.}

\email{josephelio@gmail.com}

\begin{document}

	\begin{abstract}
		This paper uses W. M. Schmidt's idea formulated in 1967 to generalise the classical theory of Diophantine approximation to subspaces of $\R^n$. Given two subspaces of $\mathbb R^n$ $A$ and $B$ of respective dimensions $d$ and $e$ with $d+e\leqslant n$, the proximity between $A$ and $B$ is measured by $t=\min(d,e)$ canonical angles $0\leqslant \theta_1\leqslant \cdots\leqslant \theta_t\leqslant \pi/2$; we set $\psi_j(A,B)=\sin\theta_j$. If $B$ is a rational subspace, his complexity is measured by its height $H(B)=\mathrm{covol}(B\cap\mathbb Z^n)$. We denote by $\muexpA nAej$ the exponent of approximation defined as the upper bound (possibly equal to $+\infty$) of the set of $\beta>0$ such that for infinitely many rational subspaces $B$ of dimension $e$, the inequality $\psi_j(A,B)\leqslant H(B)^{-\beta}$ holds. We are interested in the minimal value $\muexp ndej$ taken by $\muexpA nAej$ when $A$ ranges through the set of subspaces of dimension $d$ of $\R^n$ such that for all rational subspaces $B$ of dimension $e$ one has $\dim (A\cap B)<j$. We show that if $A$ is included in a rational subspace $F$ of dimension $k$, its exponent in $\mathbb R^n$ is the same as its exponent in $\mathbb R^k$ via a rational isomorphism $F\to\mathbb R^k$. This allows us to deduce new upper bounds for $\muexp ndej$. We also study the values taken by $\muexpA nAee$ when $A$ is a subspace of $\mathbb R^n$ satisfying $\dim(A\cap B)<e$ for all rational subspaces $B$ of dimension $e$.
	\end{abstract}

\maketitle

\section{Introduction}

Diophantine approximation in its classical sense studies how well points of $\R^n$ can be approximated by rational points. We will focus here on a different but related problem, stated by W. M. Schmidt in 1967 (see \cite{schmidt67}), which studies the approximation of subspaces of $\R^n$ by rational subspaces. The results exposed here can be found with extended details in my Ph.D. thesis (see \cite{joseph21}, chapters 5 and 6).

Let us say that a subspace of $\R^n$ is \emph{rational} whenever it admits a basis of vectors of $\Q^n$; let us denote by $\mathfrak R_n(e)$ the set of rational subspaces of dimension $e$ of $\R^n$. A subspace $A$ of $\R^n$ is said to be \emph{$(e,j)$-irrational} whenever for all $B\in\mathfrak R_n(e)$, $\dim(A\cap B)<j$; let us denote by $\mathfrak I_n(d,e)_j$ the set of all $(e,j)$-irrational subspaces of dimension $d$ of $\R^n$. Notice that $A\in\mathfrak I_n(d,e)_1$ if, and only if, for any $B\in\mathfrak R_n(e)$: $A\cap B=\{0\}$.

In order to formulate the problems we will consider, we need a notion of \emph{complexity} for a rational subspace and a notion of \emph{proximity} between two subspaces of $\R^n$.

Let $B\in\mathfrak R_n(e)$ and $\Xi=(\xi_1,\ldots,\xi_N)\in\Z^N$, with $N=\binom ne$, be a vector in the class of Pl\"ucker coordinates of $B$. Let us define the \emph{height} of $B$ to be:
\[H(B)=\norme\Xi/\gcd(\xi_1,\ldots,\xi_N)\]
where $\norme\cdot$ stands for the Euclidean norm. In particular, when $\Xi$ has setwise coprime coordinates: $H(B)=\norme\Xi$. 

We will also make use of an equivalent definition of the height of a rational subspace. Given vectors $X_1,\ldots,X_e\in\R^n$, let us denote by $M\in\M_{n,e}(\R)$ the matrix whose $j$-th column is $X_j$ for $j\in\{1,\ldots,e\}$. The \emph{generalised determinant} of the family $(X_1,\ldots,X_e)$ is defined as $D(X_1,\ldots,X_e)=\sqrt{\det(\transp MM)}$. The following result establishes a link between the generalised determinant and the height of a rational subspace (see Theorem 1 of \cite{schmidt67}).
\begin{theorem}[Schmidt, 1967]\label{th_def_equiv_hauteur_vaoribibgipn}
Let $B\in\mathfrak R_n(e)$ and $(X_1,\ldots,X_e)$ be a basis of $B\cap\Z^n$. Then 
\[H(B)=D(X_1,\ldots,X_e).\]
\end{theorem}

For $X,Y\in\R^n\setminus\{0\}$, let us define a measure of the distance between these two vectors by $\psi(X,Y)=\sin\widehat{(X,Y)}=\norme{X\wedge Y}\cdot \norme{X}^{-1}\cdot \norme{Y}^{-1}$, where $\R^n$ and is endowed with the standard Euclidean norm $\norme\cdot$, $\wedge\colon\R^n\times\R^n\to\Lambda^2(\R^n)$ stands for the exterior product on $\R^n$, and the Euclidean norm is naturally extended to $\Lambda^2(\R^n)$ so that $\norme{X\wedge Y}$ is the area of the parallelogram spanned by $X$ and $Y$. Let us define by induction $t=\min(d,e)$ angles between two subspaces $A$ and $B$ of $\R^n$ of respective dimensions $d$ and $e$. The first one is defined as
\[\psi_1(A,B)=\min_{\substack{X\in A\setminus\{0\}\\Y\in B\setminus\{0\}}}\psi(X,Y)\]
and let $X_1$ and $Y_1$ be two unitary vectors such that $\psi_1(A,B)=\psi(X_1,Y_1)$. Let $j\in\{1,\ldots,t-1\}$ and assume that the first $j$ angles $\psi_1(A,B),\ldots,\psi_j(A,B)$ have been constructed together with couples of vectors $(X_1,Y_1),\ldots,(X_j,Y_j)\in A\times B$ such that $\psi_\ell(A_\ell,B_\ell)=\psi(X_\ell,Y_\ell)$ for $\ell\in\{1,\ldots,j\}$. Let $A_j$ and $B_j$ be two subspaces of $A$ and $B$ respectively, such that $A=\span(X_1,\ldots,X_j)\overset\perp\oplus A_j$ and 
$B=\span(Y_1,\ldots,Y_j)\overset\perp\oplus B_j$. The $(j+1)$-th angle is then defined as
\[\psi_{j+1}(A,B)=\min_{\substack{X\in A_j\setminus\{0\}\\Y\in B_j\setminus\{0\}}}\psi(X,Y),\]
and let $X_{j+1}$ and $Y_{j+1}$ be two unitary vectors such that $\psi_{j+1}(A,B)=\psi(X_{j+1},Y_{j+1})$.

According to Theorem 4 of \cite{schmidt67}, there exist orthonormal bases $(X_1,\ldots,X_d)$ and $(Y_1,\ldots,Y_e)$ of $A$ and $B$ respectively, such that for all $(i,j)\in\{1,\ldots,d\}\times\{1,\ldots,e\}$, $X_i\cdot Y_j=\delta_{i,j}\cos\theta_i$, where $\delta$ is the Kronecker delta, the $\theta_\ell$ are real numbers such that $0\le \theta_t\le \cdots\le \theta_1\le 1$, and $\cdot$ is the canonical scalar product on $\R^n$; notice that $\psi_j(A,B)=\sin\theta_j$. The angles defined between $A$ and $B$ are canonical since the numbers $\theta_1,\ldots,\theta_t$ does not depend on the choice of the bases $(X_1,\ldots,X_d)$ and $(Y_1,\ldots,Y_e)$ and are invariant under the application of an orthogonal transformation on $A$ and $B$ simultaneously. 

Let us now formulate the generalisation to the classical Diophantine approximation problem. Let $n\ge 2$, $d,e\in\{1,\ldots,n-1\}$ be such that $d+e\le n$, $j\in\{1,\ldots,\min(d,e)\}$. For $A\in\mathfrak I_n(d,e)_j$, let $\muexpA nAej$ be the upper bound in $[0,+\infty]$ of all $\beta>0$ such that
\[\psi_j(A,B)\le \frac 1{H(B)^\beta}\]
holds for infinitely many $B\in\mathfrak R_n(e)$. Let 
\[\muexp ndej=\inf_{A\in\mathfrak I_n(d,e)_j}\muexpA nAej.\]
The determination of $\muexp ndej$ in terms of $n$, $d$, $e$ and $j$ is still an open problem. Some partial results are known (see \cite{schmidt67}, Theorems 12, 13, 15, 16 and 17; \cite{moshchevitin20}, Satz 2; \cite{saxce20}, Theorem 9.3.2; \cite{joseph21bis}, Theorems 1.5, 1.6, 1.7, 1.8 and Corollary 1.1). In this paper, new upper bounds on $\muexp ndej$ will be proved in Propositions 3.1 and 3.2.

The other problem tacked by this paper is the determination of the set $\muexpA n{\mathfrak I_n(d,e)_j}ej$ in terms of $(n,d,e,j)$, \emph{i.e.} the set of values taken by $\muexpA nAej$ for $A\in\mathfrak I_n(d,e)_j$. \\

One of the main results of the present paper, which allows us to improve on several known upper bounds for $\muexp ndej$, is the following theorem and its corollary below.
\begin{theorem}\label{th_inclusion_sev_rationnel_apeivpinpiaenv}
Let $n\ge 2$ and $k\in\{2,\ldots,n\}$. Let $d,e\in\{1,\ldots,k-1\}$ be such that $d+e\le k$, and $j\in\{1,\ldots,\min(d,e)\}$. Let $A$ be a subspace of dimension $d$ of $\R^n$, such that there exists a subspace $F\in\mathfrak R_n(k)$ such that $A\subset F$. Let us denote by $\varphi\colon F\to\R^k$ a rational isomorphism and let $\tilde A=\varphi(A)$, which is a subspace of dimension $d$ of $\R^k$.
Let us assume that for any rational subspace $B'$ of dimension $e$ contained in $F$, one has 
\begin{equation}\label{condition_dirr_sur_A_ameirnmoafvonsd}
\dim (A\cap B')<j.
\end{equation}
Then $A\in\mathfrak I_n(d,e)_j$, $\tilde A\in\mathfrak I_k(d,e)_j$ and 
\[\muexpA nAej=\muexpA k{\tilde A}ej.\]
\end{theorem}
One can notice that Hypothesis \eqref{condition_dirr_sur_A_ameirnmoafvonsd} of Theorem \ref{th_inclusion_sev_rationnel_apeivpinpiaenv} is \emph{a priori} a weak version of the hypothesis $A\in\mathfrak I_n(d,e)_j$ (\emph{i.e.} $\dim(A\cap B)<j$ for all $B\in\mathfrak R_n(e)$); Theorem \ref{th_inclusion_sev_rationnel_apeivpinpiaenv} shows that these two hypotheses are in fact equivalent.
\begin{corollary}\label{corollaire_th_inclusionsevrationnel_piarneeinvivan}
Let $n\ge 2$ and $k\in\{2,\ldots,n\}$. Let $d,e\in\{1,\ldots,k-1\}$ be such that $d+e\le k$, and $j\in\{1,\ldots,\min(d,e)\}$. Then one has
\[\muexp ndej\le \muexp kdej.\]
\end{corollary}
This corollary leads to new upper bounds in subsection \ref{subsection_improvements_aienoeivnovin}; for instance Proposition \ref{applicitation_deduction_Mosh_inclusion_sev_rationnel_naveoimfvn} gives if $n\ge 6$, $d\in\{3,\ldots,\lfloor n/2\rfloor\}$ and $\ell\in\{1,\ldots,d\}$: $\muexp nd{\ell}1\le 2d^2/(2d-\ell)$, improving on several known upper bounds. \\

The other main result of this paper deals with the spectrum of $\muexpA n\bullet ej$ when $d=e=j$.
\begin{theorem}\label{theoreme_spectre_amoeribnefaomnv}
Let $n\ge 2$ and $\ell\in\{1,\ldots,\lfloor n/2\rfloor\}$, one has
\[\left[1+\frac 1{2\ell}+\sqrt{1+\frac{1}{4\ell^2}},+\infty\right]\subset\Big\{\muexpA nA\ell\ell,\ A\in\mathfrak I_n(\ell,\ell)_\ell\Big\}.\]
\end{theorem}
In Section \ref{section_lemmes_surhauteuretproximite_eiofnbeev} we state some lemmas on the height and the proximity, which will find use in the other sections. In Section \ref{section_inclusion_sev_raboeinboaisnv}, we will prove and use Theorem \ref{th_inclusion_sev_rationnel_apeivpinpiaenv} to deduce new upper bounds on $\muexp ndej$. Section \ref{section_spectre_oairenboifnv} is dedicated to prove Theorem \ref{theoreme_spectre_amoeribnefaomnv}, which brings a partial answer to the problem of the determination of the set $\muexpA n{\mathfrak I_n(d,e)_j}ej$; the main theorem of Section \ref{section_inclusion_sev_raboeinboaisnv} is also used in this section.

\section{Some results about the height and the proximity}\label{section_lemmes_surhauteuretproximite_eiofnbeev}

The first lemma is proved in \cite{schmidt67} (Lemma 13).
\begin{lemma}[Schmidt, 1967]\label{lemma_13_schmidt_beoifnboinfsoin}
Let $A$ and $B$ be two subspaces of $\R^n$ of dimensions $d$ and $e$ respectively, let $\varphi$ be a non-singular linear transformation of $\R^n$. There exists a constant $c(\varphi)>0$ such that for all $j\in\{1,\ldots,\min(d,e)\}$, $\psi_j(\varphi(A),\varphi(B))\le c(\varphi)\psi_j(A,B)$.
\end{lemma}
We will make use of the brief Lemma \ref{lemme_transfert_psi1_psiell_baoeoearv} below, but first, we require a lemma of Schmidt (see \cite{schmidt67}, Lemma 12), which will also find use in the proofs of Lemma \ref{inegalitesurladistance_ainreoaeonbe} and Theorem \ref{th_inclusion_sev_rationnel_apeivpinpiaenv}.
\begin{lemma}\label{lemma_12_Schmidt_vaeorinbofiso}
Let $A$ and $B$ be two subspaces of $\R^n$ of dimensions $d$ and $e$ respectively, let $j\in\{1,\ldots,\min(d,e)\}$. Then $\psi_j(A,B)$ is the smallest number $\lambda$ so there is a subspace $A_j\subset A$ so that for every $X\in A_j\setminus\{0\}$, there exists $Y\in B\setminus\{0\}$ such that $\psi(X,Y)\le \lambda$.
\end{lemma}
\begin{lemma}\label{lemme_transfert_psi1_psiell_baoeoearv}
Let $A$ and $B$ be two non-trivial subspaces of $\R^n$ such that $\dim A\le \dim B$. Then
\[\forall X\in A\setminus\{0\},\quad \psi_1(\span(X),B)\le\psi_{\dim A}(A,B).\]
\end{lemma}
\begin{proof}
Let $X\in A\setminus\{0\}$. One has
\begin{equation*}
\begin{split}
\psi_1(\span(X),B)
	&=\min_{Y\in B\setminus\{0\}}\psi(X,Y) \\
	&\le \max_{Z\in A\setminus\{0\}}\min_{Y\in B\setminus\{0\}}\psi(Z,Y) \\
	&=\min\{\varphi,\ \forall Z\in A\setminus\{0\},\quad \exists Y\in B\setminus\{0\},\quad \psi(Z,Y)\le \varphi\} \\
	&=\psi_{\dim A}(A,B)
\end{split}
\end{equation*}
using Lemma \ref{lemma_12_Schmidt_vaeorinbofiso}.
\end{proof}
Now, we prove a result on the behaviour of the height of a rational subspace when applying a rational morphism.
\begin{lemma}\label{inegalitesurlahauteur_aroibvaoivbnoi}
Let $n\ge 3$ and $e,p\in\{1,\ldots,n\}$; let $B\in\mathfrak R_e(n)$ and $F$ be two rational subspaces of $\R^n$ such that $B\subset F$; let $\varphi\colon F\to \R^p$ be a rational morphism such that $\dim \varphi (B)=\dim B$. There exists a constant $c(\varphi)>0$, depending only on $\varphi$, such that
\[H(\varphi(B))\le c(\varphi)H(B).\]
\end{lemma}
\begin{proof}
Let us extend $\varphi$ to a rational endomorphism of $\R^n$ by extending its codomain from $\R^p$ to $\R^n$, and by letting $\varphi(x)=x$ for all $x\in F^\perp$.

First, assume that the subspace $B$ is a rational line $L$. Let $\xi=(\xi_1,\ldots,\xi_n)\in\Z^n$ be such that $\gcd(\xi_1,\ldots,\xi_n)=1$ and $L=\span(\xi)$. One has $\varphi(L)=\span(\varphi(\xi))$, and there exists $c_1(\varphi)>0$ independent of $\xi$ such that $\norme {\varphi(\xi)}\le c_1(\varphi) \norme \xi$. Let $(\zeta_1,\ldots,\zeta_n)\in\Q^n$ be the coordinates of $\varphi(\xi)$. Since $\varphi\in\M_n(\Q)$, there exists $k\in\Z\setminus\{0\}$ such that $k\varphi\in\M_n(\Z)$, so $k\varphi(\xi)\in\Z^n$, and then for all $i\in\{1,\ldots,n\}$,  $k\zeta_i\in\Z$. Let $\mathfrak b$ be the fractional ideal spanned by the $\zeta_i$, one has $k\mathfrak b=k(\zeta_1\Z+\cdots+\zeta_n\Z)\subset \Z$, thus $kN(\mathfrak b)=N(k\mathfrak b)\ge 1$. Therefore, using the generalised definition of the height of a rational subspace (see \cite{schmidt67}, Equation (1) page 432):
\[H(\varphi(L))=N(\mathfrak b)^{-1} \norme{\varphi(\xi)}\le k c_1(\varphi) \norme{\xi}=c(\varphi) \norme{\xi}=c(\varphi)H(L)\]
with $c(\varphi)=kc_1(\varphi)$. 

Let us extend the result to a rational subspace $B$ of dimension $e$. Let $N=\binom ne$ and $B^*$ be the rational line of $\R^N$ spanned by the Pl\"ucker coordinates of $B$. The hypothesis of the lemma gives $\dim\varphi(B)=\dim B$, so the rational line $\varphi(B)^*$ spanned by the Plücker coordinates of $\varphi(B)$ also belongs to $\R^N$. Notice that $H(B)=H(B^*)$ and $H(\varphi(B))=H(\varphi(B)^*)$. Let us denote by $S\in\M_n(\Q)$ the matrix of $\varphi$ in the canonical basis of $\R^n$. Then $\Lambda^e(S)\in\M_N(\Q)$, the matrix formed with all $e\times e$ minors of $S$ in lexicographic order, is the matrix of $\varphi^{(e)}$, the $e$-th compound of $\varphi$, in the canonical basis of $\R^N$. One has $\varphi(B)^*=\varphi^{(e)}(B^*)$ (see \cite{schmidt67}, page 433), so $H(\varphi(B)^*)=H(\varphi^{(e)}(B^*))$. This falls into the case of dimension $1$ in $\R^N$, therefore the first part of the proof concludes and gives a constant $c^{(e)}(\varphi)$.

Notice that the constant $c(\varphi)$ does not depend on $e$ by taking $c(\varphi)=\max_{1\le e\le n} {c}^{(e)}(\varphi)$.
\end{proof}

\section{Inclusion in a rational subspace}\label{section_inclusion_sev_raboeinboaisnv}

Here, we will focus on the case where the subspace we are trying to approach is included in a rational subspace. This will lead to several improvements on the known upper bounds for $\muexp ndej$.

First, we will state in Subsection \ref{subsection_improvements_aienoeivnovin} the new results that can be deduced from Corollary \ref{corollaire_th_inclusionsevrationnel_piarneeinvivan} of Theorem \ref{th_inclusion_sev_rationnel_apeivpinpiaenv}. Then, we will establish two lemmas in Subsection \ref{subsection_proofsofthelemmas_pevinevnivon} which will find use in the proof of the main result, Theorem \ref{th_inclusion_sev_rationnel_apeivpinpiaenv}, in Subsection \ref{subsection_preuve_th_inclusionsevrationnel_aeobindfv}.

\subsection{Improvements on some upper bounds}\label{subsection_improvements_aienoeivnovin}

First, the upper bound $\muexp 5221\le 4$ given by Theorem 16 of \cite{schmidt67} is improved (Theorem 12 of \cite{schmidt67} gives $\muexp 5221\ge 20/9$, so an equality is not obtained here).
\begin{proposition}\label{amelioration_R5_grace_inclusion_R4_aoibvoeubv}
One has
\[\muexp 5221\le 3.\]
\end{proposition}
The proof of Proposition \ref{amelioration_R5_grace_inclusion_R4_aoibvoeubv} requires Theorem 1.5 of \cite{joseph21bis}: $\muexp 4221=3$.

In a very similar fashion, the following proposition is deduced from Theorem 1.7 of \cite{joseph21bis}: $\muexp {2d}d{\ell}1\le 2d^2/(2d-\ell)$ for $d\ge 2$ and $\ell\in\{1,\ldots,d\}$.
\begin{proposition}\label{applicitation_deduction_Mosh_inclusion_sev_rationnel_naveoimfvn}
Let $n\ge 6$, $d\in\{3,\ldots,\lfloor n/2\rfloor\}$ and $\ell\in\{1,\ldots,d\}$, one has
\[\muexp nd{\ell}1\le\frac{2d^2}{2d-\ell}.\]
\end{proposition}
This proposition improves on several upper bounds when $n$ is close to $2d$ and $\ell$ is close to $d$; for instance, the upper bounds $\muexp 9431\le 7$ and $\muexp{30}{10}81\le 18$ proved by Schmidt (see \cite{schmidt67}, Theorem 16) are improved respectively to $32/5$ and $50/3$.

\subsection{Two useful lemmas to prove Theorem \ref{th_inclusion_sev_rationnel_apeivpinpiaenv}}\label{subsection_proofsofthelemmas_pevinevnivon}

Let us prove two lemmas which will find use in the proof of Theorem \ref{th_inclusion_sev_rationnel_apeivpinpiaenv} in Subsection \ref{subsection_preuve_th_inclusionsevrationnel_aeobindfv}. The first lemma studies how the proximity between two subspaces behaves when applying a projection. This proof follows the ideas of Lemma 13 of \cite{schmidt67}, though the endomorphism is not assumed to be invertible here.
\begin{lemma}\label{inegalitesurladistance_ainreoaeonbe}
Let $\mathcal R$ be a non-empty subset of $\R^n$ such that $\mathcal R\cap F^\perp=\emptyset$ and such that there exists a constant $c>0$ satisfying
\begin{equation}\label{hyp_lemme_proximite_projection_aoeibnoivnc}
\forall X\in\mathcal R,\quad \norme{p_F^\perp(X)}\ge c\norme X
\end{equation}
where $p_F^\perp$ is the orthogonal projection onto $F$. Let $D$ be a subspace of $\R^n$ such that $\dim D\ge j$ and $D\subset \mathcal R\cup\{0\}$. Then there exists a constant $c'>0$ depending only on $c$ such that
\[\psi_j(A,D)\ge c'\psi_j(A,p_F(D)).\]
\end{lemma}
\begin{proof}
Hypothesis \eqref{hyp_lemme_proximite_projection_aoeibnoivnc} gives a constant $c>0$ such that for all $X\in\mathcal R$, $c\norme X\le \norme{p_F^\perp(X)}$. In particular $c\le 1$ since $p_F^\perp$ is a projection, therefore we may assume that $F\setminus\{0\}\subset \mathcal R$ since $\norme{p_F^\perp(X)}=\norme X$ for any $X\in F$. Our first goal is to show that there exists a constant $c_1>0$ (depending only on $c>0$), such that
\begin{equation}\label{inegalite_projection_aerinbaoeinvipnczd}
\forall X\in F\setminus\{0\},\quad \forall Y\in\mathcal R,\quad \psi(X,p_F^\perp(Y))\le c_1\psi(X,Y).
\end{equation}
Let $X\in F\setminus\{0\}$ and $Y\in\mathcal R$. Without loss of generality, assume that $\norme X=\norme Y=1$ and $X\cdot Y\ge 0$. One has $\psi^2(X,p_F^\perp(Y))=(\norme{p_F^\perp(Y)}^2-(X\cdot p_F^\perp(Y))^2)/\norme{p_F^\perp(Y)}^2$. Let $\lambda=\norme{p_F^\perp(Y)}$ and notice that $0\le(X\cdot p_F^\perp(Y)-\lambda)^2$ leads to $\lambda^2-(X\cdot p_F^\perp(Y))^2\le 2(\lambda^2-\lambda(X\cdot p_F^\perp(Y)))$. Thus,
\[\psi^2(X,p_F^\perp(Y))\le2\frac {\norme{p_F^\perp(Y)}-X\cdot p_F^\perp(Y)}{\norme{p_F^\perp(Y)}}.\]
Using Hypothesis \eqref{hyp_lemme_proximite_projection_aoeibnoivnc}, $\psi^2(X,p_F^\perp(Y))\le \displaystyle\frac 2c\left(\norme{p_F^\perp(Y)}-X\cdot p_F^\perp(Y)\right)$, but
\[\frac 12\left(\norme{X-p_F^\perp(Y)}^2-1-\norme{p_F^\perp(Y)}^2\right)=-X\cdot p_F^\perp(Y),\]
so
\begin{equation}\label{majoration_oeztinoanrevzomefbnoubf}
\psi^2(X,p_F^\perp(Y))\le\frac 1{c}\norme{p_F^\perp(X-Y)}^2-\frac1{c}\left(\norme{p_F^\perp(Y)}-1\right)^2\le\frac {1}{c}\norme{X-Y}^2.
\end{equation}

Let us mention an elementary geometric claim.
\begin{claim}\label{claim_minoration_psi_aeroibneionv}
Let $U$ and $V$ be unitary vectors such that $U\cdot V\ge 0$. One has $\psi(U,V)\ge \frac{\sqrt 2}2\norme{U-V}$.
\end{claim}
\begin{proof}[Proof of Claim \ref{claim_minoration_psi_aeroibneionv}.]
Let $p_{\span(V)}^\perp$ be the orthogonal projection onto $\span(V)$, $\alpha=\Vert U-p_{\span(V)}^\perp(U)\Vert$ and $\beta=\Vert V-p_{\span(V)}^\perp(U)\Vert$. One has $\norme{U-V}^2=\alpha^2+\beta^2$, and since $U$ is unitary: $\psi(U,V)=\psi(U,p_{\span(V)}^\perp(U))=\Vert U-p_{\span(V)}^\perp(U)\Vert=\alpha$. Moreover, $U\cdot V\ge 0$, so $1=\norme{U}^2=(1-\beta)^2+\alpha^2$, hence there exists $\theta\in[0,\pi/2]$ such that $1-\beta=\cos\theta$ and $\alpha=\sin\theta$. Since $1-\cos\theta\le \sin\theta$, we obtain $\beta\le \alpha$, and finally $\norme{U-V}^2\le 2\alpha^2=2\psi(U,V)^2$.
\end{proof}

Since $X\cdot Y\ge 0$ here, Claim \ref{claim_minoration_psi_aeroibneionv} gives $\psi(X,Y)\ge\frac {\sqrt 2}2\norme{X-Y}$, so with Inequality \eqref{majoration_oeztinoanrevzomefbnoubf}, it yields
\[\psi(X,p_F^\perp(Y))\le c_1\psi(X,Y)\]
which is the desired Inequality \eqref{inegalite_projection_aerinbaoeinvipnczd}, with $c_1=\sqrt{2/c}$. 

For the second part of the proof, Lemma \ref{lemma_12_Schmidt_vaeorinbofiso} tells us that there exists a subspace $A_j\subset A$ of dimension $j$, such that for all $X\in A_j\setminus\{0\}$, there exists $Y\in D\setminus\{0\}$ such that $\psi(X,Y)\le \psi_j(A,D)$. Let $X\in A_j\setminus\{0\}$ and $Y\in D\setminus\{0\}$ be such that $\psi(X,Y)\le \psi_j(A,D)$. Since $X\in A_j\subset A\subset F$ and $Y\in D\setminus\{0\}\subset \mathcal R$, one can use Inequality \eqref{inegalite_projection_aerinbaoeinvipnczd} to get $\psi(X,p_F^\perp(Y))\le c_1\psi(X,Y)\le c_1\psi_j(A,D)$. Thus, $Y'=p_F^\perp(Y)\in p_F^\perp(D)$ is a non-zero vector such that $\psi(X,Y')\le c_1\psi_j(A,D)$; therefore
\[\forall X\in A\setminus\{0\},\quad \exists Y'\in p_F^\perp(D)\setminus\{0\},\quad \psi(X,Y')\le c_1\psi_j(A,D).\]
According to Lemma \ref{lemma_12_Schmidt_vaeorinbofiso}, $\psi_j(A,p_F^\perp(D))$ is the smallest number having this property, so
\[\psi_j(A,p_F^\perp(D))\le c_1\psi_j(A,D).\]
\end{proof}

The second lemma shows that one can choose rational subspaces approaching $A$ which intersect $F^\perp$ trivially. 
\begin{lemma}\label{lemme_intersection_orthogonal_est_vide_aorfbaovbd}
Under the hypothesis of Theorem \ref{th_inclusion_sev_rationnel_apeivpinpiaenv}, for all $\alpha<\muexpA nAej$, there exists a sequence $(B_N)_{N\in\N}$ of rational subspaces of $\R^n$ of dimension $e$, pairwise distinct, such that for any $N$ large enough: $B_N\cap F^\perp=\{0\}$ and $\psi_j(A,B_N)\le H(B_N)^{-\alpha}$.
\end{lemma}

\begin{proof}[Proof of Lemma \ref{lemme_intersection_orthogonal_est_vide_aorfbaovbd}.]
Let $\alpha'$ be such that $\alpha<\alpha'<\muexpA nAej$. By definition of $\muexpA nAej$, there exists a sequence $(B_N)_{N\in\N}$ of rational subspaces of $\R^n$ of dimension $e$, pairwise distinct, such that for any $N$ large enough: $\psi_j(A,B_N)\le H(B_N)^{-\alpha'}$.

Let $(g_1,\ldots,g_{n-k})$ be a linearly independent family of vectors of $\Q^n$ such that $\R^n=F\oplus\span(g_1,\ldots,g_{n-k})$. For $\ell\in\N^*$, let $\mathcal P_m(\ell)$ be the set of subsets with $m$ elements of $\{1,\ldots,\ell\}$. Let us show by induction on $\ell\in\{n-k,\ldots,n\}$ that there exist vectors $g_{n-k+1},\ldots,g_\ell$ of $\Q^n$ such that
\begin{equation}\label{rec_construction_gi_inclusionsevrationnel_aepirbniv}
\begin{cases}\dim\span(g_1,\ldots,g_\ell)=\ell,\\\forall I\in\mathcal P_{n-k}(\ell),\quad G_I\cap F=\{0\};\end{cases}
\end{equation} 
here and below, for $I\subset \{1,\ldots,\ell\}$ we let $G_I=\span\{g_i,\ i\in I\}$. Since $\dim F=k$, the initial case $\ell=n-k$ holds by definition of the $g_1,\ldots,g_{n-k}$. Let $\ell\in\{n-k,\ldots,n-1\}$, assume that the $g_i$ for $i\in\{n-k+1,\ldots,\ell\}$ have been constructed satisfying \eqref{rec_construction_gi_inclusionsevrationnel_aepirbniv}. Let
\[G=\span(g_1,\ldots,g_\ell)\cup\bigcup_{K\in\mathcal P_{n-k-1}(\ell)} \big(F\oplus G_K\big).\]
The set $G$ is an union of a finite number of subspaces of dimensions at most $n-1$, thus there exists a vector $g_{\ell+1}\in\Q^n\setminus G$. Notice that $\dim\span(g_1,\ldots,g_{\ell+1})=\ell+1$. Let us assume that there exists $I\in\mathcal P_{n-k}(\ell+1)$ such that $G_I\cap F\ne\{0\}$. Let $u\in G_I\cap F\setminus\{0\}$; by the induction hypothesis, $I\notin\mathcal P_{n-k}(\ell)$, so $\ell+1\in I$. Let us write $I$ under the form $I=K\cup\{\ell+1\}$ with $K\in\mathcal P_{n-k-1}(\ell)$. Since $u\notin F\cap G_K=\{0\}$, there exist $\alpha_{\ell+1}\ne 0$ and some $\alpha_i\in\R$ (for $i\in K$) such that $u=\alpha_{\ell+1}g_{\ell+1}+\sum_{i\in K}\alpha_ig_i$, so
\[g_{\ell+1}=\frac 1{\alpha_{\ell+1}}\left(u-\sum_{i\in K}\alpha_ig_i\right)\in F\oplus G_K,\]
which can not be by definition of $G$, because $g_{\ell+1}\notin G$. Thus, $g_{\ell+1}$ satisfies \eqref{rec_construction_gi_inclusionsevrationnel_aepirbniv} and therefore the induction is complete; it is now established that there exist vectors $g_1,\ldots,g_n$ satisfying \eqref{rec_construction_gi_inclusionsevrationnel_aepirbniv}.

For $I\in\mathcal P_{n-k}(n)$, one has $\dim G_I=n-k=\dim F^\perp$. So $G_I\oplus F=\R^n$: there exists a rational isomorphism $\rho_I\in\GL_n(\Q)$ such that ${\rho_I}_{\vert F}=\id_F$ and $\rho_I(G_I)=F^\perp$. 
Let $N\in\N$ and let us assume that
\begin{equation}\label{equation_pourtoutI_rhoIBNFperpegal0_banfpibnv}
\forall I\in\mathcal P_{n-k}(n),\quad \rho_I(B_N)\cap F^\perp\ne\{0\},
\end{equation}
which is equivalent by definition of $\rho_I$ to $B_N\cap G_I\ne\{0\}$ for any $I\in\mathcal P_{n-k}(n)$.
Let
\[J=\left\{i\in\{1,\ldots,n\},\ \exists \alpha\ne 0,\quad\exists\lambda_1,\ldots,\lambda_{i-1}\in\R,\quad \alpha g_i +\sum_{\ell=1}^{i-1} \lambda_\ell g_\ell\in B_N\right\}.\]
First, assume that $\abs J\le k$. Thus, there exists $I\in\mathcal P_{n-k}(n)$ such that $I\cap J=\emptyset$. Since $B_N\cap G_I\ne\{0\}$, there exists a non-zero vector $(\beta_i)_{i\in I}\in\R^{I}$ such that $\sum_{i\in I}\beta_i g_i\in B_N$. Let $i_0$ be the largest $i\in I$ such that $\beta_i\ne 0$. Then 
\[\sum_{i\in I}\beta_i g_i=\alpha g_{i_0}+\sum_{i=1}^{i_0-1}\beta_i g_i\]
where $\beta_i=0$ if $i\notin I$ and $\alpha=\beta_{i_0}\ne 0$. So $i_0\in I\cap J$ which can not be, therefore $\abs J>k$. The elements of $J$ give at least $k+1$ linearly independent vectors of $B_N$, which can not be since $\dim B_N\le k$. Thus, \eqref{equation_pourtoutI_rhoIBNFperpegal0_banfpibnv} is established.

Let $I\in\mathcal P_{n-k}(n)$ be such that $\rho_I(B_N)\cap F^\perp=\{0\}$. Since $\rho_I$ is invertible, Lemma \ref{lemma_13_schmidt_beoifnboinfsoin} gives a constant $c(\rho_I)>0$ such that
\[\psi_j(A,\rho_I(B_N))=\psi_j(\rho_I(A),\rho_I(B_N))\le c(\rho_I)\psi_j(A,B_N).\]
Let $c_2=\max_{I\in\mathcal P_{n-k}(n)}c(\rho_I)>0$, which is a constant independent of $B_N$; then we have $\psi_j(A,\rho_I(B_N))\le c_2\psi_j(A,B_N)$. Moreover, since $\rho_I$ is an isomorphism, $\dim (\rho_I(B_N))=\dim(B_N)=e$, so Lemma \ref{inegalitesurlahauteur_aroibvaoivbnoi} gives a constant $c'(\rho_I)>0$ such that $H(\rho_I(B_N))\le c'(\rho_I)H(B_N)$. Let $c_3=\max_{I\in\mathcal P_{n-k}(n)}c'(\rho_I)>0$ which is a constant independent of $B_N$ such that $H(\rho_I(B_N))\le c_3H(B_N)$. Therefore, for all $N$ large enough:
\[\psi_j(A,\rho_I(B_N))\le c_2\psi_j(A,B_N)\le \frac{c_2}{H(B_N)^{\alpha'}}\le \frac{c_2c_3^{-\alpha'}}{H(\rho_I(B_N))^{\alpha'}}\le \frac 1{H(\rho_I(B_N))^\alpha}.\]
\end{proof}

\subsection{Proof of Theorem \ref{th_inclusion_sev_rationnel_apeivpinpiaenv}}\label{subsection_preuve_th_inclusionsevrationnel_aeobindfv}

Let us provide a proof of the main theorem.
\begin{proof}[Proof of Theorem \ref{th_inclusion_sev_rationnel_apeivpinpiaenv}.]
First, let us prove that $A\in\mathfrak I_n(d,e)_j$ and $\tilde A\in\mathfrak I_k(d,e)_j$.

Let $B\in\mathfrak R_n(e)$. Notice that $B'=B\cap F$ is a rational subspace of dimension $e'\le e\le \dim F$. Thus, there exists a rational subspace $B''\subset F$, containing $B'$, and such that $\dim B''=e$. Hypothesis \eqref{condition_dirr_sur_A_ameirnmoafvonsd} of Theorem \ref{th_inclusion_sev_rationnel_apeivpinpiaenv} gives $\dim(A\cap B'')<j$, therefore $\dim(A\cap B')<j$. Since $A\cap B=A\cap F\cap B=A\cap B'$ because $A\subset F$, one has $\dim(A\cap B)<j$, \emph{i.e.} $A\in\mathfrak I_n(d,e)_j$.

Let $\tilde B\in\mathfrak R_k(e)$ and $B=\varphi^{-1}(\tilde B)\in\mathfrak R_n(e)$. Since $\varphi$ is an isomorphism, one has $\dim(\tilde A\cap \tilde B)=\dim(\varphi(A)\cap\varphi(B))=\dim(\varphi(A\cap B))=\dim(A\cap B)<j$ because $B\in\mathfrak R_n(e)$ and $A\in\mathfrak I_n(d,e)_j$. This shows that $\tilde A\in\mathfrak I_k(d,e)_j$. \\

Now, let us show that $\muexpA nAej\ge \muexpA k{\tilde A}ej$. Let $\alpha<\muexpA k{\tilde A}ej$. There exists a sequence $(\tilde B_N)_{N\ge 0}$ of rational subspaces of $\R^k$ of dimension $e$, pairwise distinct, such that for all $N$ large enough: $\psi_j(\tilde A,\tilde B_N)\le H(\tilde B_N)^{-\alpha}$.
For all $N\in\N$, let $B_N=\varphi^{-1}(\tilde B_N)\in\mathfrak R_n(e)$ because $\varphi$ is a rational isomorphism. According to Lemma \ref{inegalitesurlahauteur_aroibvaoivbnoi}, there exists a constant $c_{\varphi^{-1}}$ such that for all $N\in\N$, $H(B_N)=H(\varphi^{-1}(\tilde B_N))\le c_{\varphi^{-1}}H(\tilde B_N)$. Using Lemma \ref{lemma_13_schmidt_beoifnboinfsoin}, there exists a constant $c_{\varphi^{-1}}'>0$ such that $\psi_j(A,B_N)=\psi_j(\varphi^{-1}(\tilde A),\varphi^{-1}(\tilde B_N))\le c_{\varphi^{-1}}'\psi_j(\tilde A,\tilde B_N)$. Therefore, for $N$ large enough,
\[\psi_j(A,B_N)\le c_{\varphi^{-1}}'\psi_j(\tilde A,\tilde B_N)\le \frac{c_{\varphi^{-1}}'}{H(\tilde B_N)^\alpha}\le \frac {c_1}{H(B_N)^\alpha},\]
with $c_1>0$ depending only on $\varphi$. Since the $B_N$ are pairwise distinct, $\muexpA nAej\ge\alpha$, and since this is true for all $\alpha<\muexpA k{\tilde A}ej$, one has
\[\muexpA nAej\ge \muexpA k{\tilde A}ej.\]

Finally, let us establish that $\muexpA nAej\le \muexpA k{\tilde A}ej$. Let $\alpha<\muexpA nAej$. Lemma \ref{lemme_intersection_orthogonal_est_vide_aorfbaovbd} gives us a sequence $(B_N)_{N\in\N}$ of rational subspaces of $\R^n$ of dimension $e$, pairwise distinct, such that for all $N$ large enough:
\begin{equation}\label{def_des_B_N_inclusionsevrationnel_aboemfbv}
B_N\cap F^\perp=\{0\}\quad\text{ and }\quad\psi_j(A,B_N)\le\frac{1}{H(B_N)^\alpha}.
\end{equation}
Let $N\in\N$ large enough; let us denote by $p_F^\perp$ the orthogonal projection onto $F$. Since $F\in\mathfrak R_n(k)$, $p_F^\perp$ is a rational endomorphism of $\R^n$. Let $B_N'=p_F^\perp(B_N)$; since $B_N$ is a rational subspace, $B_N'$ also is. Let $\mathcal R$ be the set of all non-zero vectors of $\R^n$ which form an angle lower than $\pi/4$ with the subspace $F$:
\[\mathcal R=\left\{X\in\R^n\setminus\{0\},\ \psi_1(F,\span(X))< \frac{\sqrt 2}{2}\right\}.\]
Notice that $\mathcal R\cap F^\perp=\emptyset$. Let $X\in\mathcal R$ and $Y=X-p_F^\perp(X)$; then $\norme{p_F^\perp(X)}^2+\norme Y^2=\norme X^2$. Since $X\in\mathcal R$, one has $\psi(X,F)=\psi(X,p_F^\perp(X))=\norme Y / \norme X <\sqrt 2/2$, so $\norme Y\le(\sqrt 2/2)\norme X$. Thus, $\norme{p_F^\perp(X)}^2=\norme X^2-\norme Y^2\ge \norme X^2-\frac 12 \norme X^2=\frac 12 \norme X^2$,
so the set $\mathcal R$ satisfies Hypothesis \eqref{hyp_lemme_proximite_projection_aoeibnoivnc} of Lemma \ref{inegalitesurladistance_ainreoaeonbe} with $c=\sqrt 2/2$.

According to Lemma \ref{lemma_12_Schmidt_vaeorinbofiso}, $\psi_j(A,B_N)$ is the smallest number $\lambda$ for which there exists a subspace $B_{N,j}$ of dimension $j$ such that for every $Y\in B_{N,j}\setminus\{0\}$, there is a vector $X\in A\setminus\{0\}$ such that $\psi(X,Y)\le \lambda$. Let us fix such a subspace $B_{N,j}$. Since $N$ is assumed to be large enough, $\psi_j(A,B_N)\le 1/2$ can be assumed. Therefore, for all $Y\in B_{N,j}\setminus\{0\}$, there exists a non-zero vector $X\in A\subset F$ such that $\psi(X,Y)\le \psi_j(A,B_N)\le 1/2$, so $\psi_1(F,\span(Y))\le 1/2<\sqrt 2/2$, hence $Y\in\mathcal R$. Thus, for all $N$ large enough: $B_{N,j}\setminus\{0\}\subset \mathcal R$.

Applying Lemma \ref{inegalitesurladistance_ainreoaeonbe} provides a constant $c_4>0$ which depends neither on $A$ nor on $B_N$, such that
\begin{equation}\label{minoration_psijAB_N_gaepifbaiondv}
\psi_j(A,B_N)=\psi_j(A,B_{N,j})\ge c_4\psi_j(A,p_F^\perp(B_{N,j}))\ge c_4\psi_j(A,B_N')
\end{equation}
because $B_N'=p_F^\perp(B_N)\supset p_F^\perp(B_{N,j})$. Since $B_N\cap F^\perp=\{0\}$, $\dim B_N'=e$; since $B_N'\subset F$, let $\tilde B_N=\varphi(B_N')\in\mathfrak R_k(e)$.  Using Lemma \ref{lemma_13_schmidt_beoifnboinfsoin}, there exists a constant $c_{\varphi}>0$ such that $\psi_j(\varphi(A),\varphi(B_N'))\le c_\varphi\psi_j(A,B_N')$. Let $\beta>\muexpA k{\tilde A}ej$; using Inequality \eqref{minoration_psijAB_N_gaepifbaiondv} yields that for all $N$ large enough (in terms of $\beta$): 
\[\psi_j(A,B_N)\ge c_4\psi_j(A,B_N') \ge c_4c_{\varphi}^{-1}\psi_j(\varphi(A),\varphi(B_N'))=c_4c_{\varphi}^{-1}\psi_j(\tilde A,\tilde B_N)\ge \frac{c_5}{H(\tilde B_N)^\beta}\]
with $c_5>0$. According to Lemma \ref{inegalitesurlahauteur_aroibvaoivbnoi}, there exists a constant $c_6>0$ such that $H(\tilde B_N)=H(\varphi(B_N'))\le c_6H(B_N')$, so $\psi_j(A,B_N)\ge c_7 H(B_N')^{-\beta}$ with $c_7>0$. Since $B_N\cap F^\perp=\{0\}$, $\dim (p_F^\perp(B_N))=\dim(B_N)$. Therefore, Lemma \ref{inegalitesurlahauteur_aroibvaoivbnoi} can be used again to obtain a constant $c_8>0$ such that $H(B_N')=H(p_F^\perp(B_N))\le c_8 H(B_N)$. With Inequality \eqref{def_des_B_N_inclusionsevrationnel_aboemfbv}, there exists a constant $c_9>0$ such that
\[\frac 1{H(B_N)^\alpha}\ge \psi_j(A,B_N)\ge \frac{c_9}{H(B_N)^\beta}.\]

Finally, since $H(B_N)$ tends to infinity when $N\to+\infty$, $\alpha\le\beta$. Because this is true for all $\alpha<\muexpA nAej$ and for all $\beta>\muexpA k{\tilde A}ej$, one has
\[\muexpA nAej\le \muexpA k{\tilde A}ej.\]
\end{proof}

\section{The spectrum of $\muexpA n{\bullet}{\ell}\ell$}\label{section_spectre_oairenboifnv}

In this section, progress will be made on the determination of the spectrum of $\muexpA n{\bullet}{\ell}\ell$ over $\mathfrak I_n(\ell,\ell)_\ell$, \emph{i.e.} on the set $\muexpA n{\mathfrak I_n(\ell,\ell)_\ell}{\ell}\ell$. The main result is Theorem \ref{theoreme_spectre_amoeribnefaomnv}: let $n\ge 2$ and $\ell\in\{1,\ldots,\lfloor n/2\rfloor\}$, one has
\[\left[1+\frac 1{2\ell}+\sqrt{1+\frac{1}{4\ell^2}},+\infty\right]\subset\Big\{\muexpA nA\ell\ell,\ A\in\mathfrak I_n(\ell,\ell)_\ell\Big\}.\]
\begin{remark}
In \cite{saxce20}, N. de Saxcé shows that $\muexp n\ell\ell\ell\le n/(\ell(n-\ell))$. It would be interesting to establish that
\[\left[\frac{n}{\ell(n-\ell)},+\infty\right]\subset\Big\{\muexpA nA\ell\ell,\ A\in\mathfrak I_n(\ell,\ell)_\ell\Big\}.\]
\end{remark}

To prove Theorem \ref{theoreme_spectre_amoeribnefaomnv}, it is first assumed that $n=2\ell$, and that $\beta<+\infty$ is fixed in the interval of Theorem \ref{theoreme_spectre_amoeribnefaomnv}; a subspace approximated exactly to the order $\beta$ is constructed. First, we establish that the subspace constructed satisfies $A\in\mathfrak I_{2\ell}(\ell,\ell)_\ell$, then we show that $\muexpA {2\ell}A\ell\ell\ge\beta$, and finally that $\muexpA{2\ell}A\ell\ell\le\beta$. The result will be finally extended to the case $n>2\ell$ with Theorem \ref{th_inclusion_sev_rationnel_apeivpinpiaenv}, and to the case $\beta=+\infty$.

\begin{proof}[Proof of Theorem \ref{theoreme_spectre_amoeribnefaomnv}.]
Let $\ell\ge 1$ be an integer and $n=2\ell$. Let $\beta$ be a real number such that
\begin{equation}\label{condition_sur_beta_aomeibnoaidbnv}
\beta\ge1+\frac 1{2\ell}+\sqrt{1+\frac1{4\ell^2}}.
\end{equation}
The goal is to construct $A\in\mathfrak I_n(\ell,\ell)_\ell$, a subspace $(\ell,\ell)$-irrational of $\R^n$ such that $\muexpA nA\ell\ell=\beta$. Let $\alpha=\ell\beta$; for all $(i,j)\in\{1,\ldots,\ell\}^2$ let
\[\xi_{i,j}=\sum_{k=0}^\infty \frac{e^{(i,j)}_k}{\theta^{\lfloor\alpha^k\rfloor}}\]
where the $(e^{(i,j)}_k)_{k\in\N}$ are sequences which are yet to be determined, with values in $\{1,2\}$ if $i\ne j$ and with values in $\{2\ell,2\ell+1\}$ if $i=j$, and where $\theta$ is the smallest prime number such that
\begin{equation}\label{hyp_sur_theta_amorinvamnoavcbo}
\theta>(n+1)^{n/2}\left(\frac n2\right)!=\ell!\,(2\ell+1)^{\ell}.
\end{equation}
\begin{remark}
The fact that $\theta$ is chosen to be the \emph{smallest} such number does not have any other purpose but to allow the constants not to depend on $\theta$. In practice, any prime number $\theta$ satisfying Inequality \eqref{hyp_sur_theta_amorinvamnoavcbo} would work. Hypothesis \eqref{hyp_sur_theta_amorinvamnoavcbo} on $\theta$ and the fact that the sequences $(e_k^{(i,j)})_{k\in\N}$ belong to $\{1,2\}$ or $\{2\ell,2\ell+1\}$ will be used in the proof of Claim \ref{lemmeminorationhauteurdesBN_aeoinmoaeivnv} to establish a lower bound on the height of the rational subspaces constructed below.
\end{remark}

Let $I_\ell$ be the identity matrix of $\M_\ell(\R)$, $M_\xi=(\xi_{i,j})_{(i,j)\in\{1,\ldots,\ell\}^2}\in\M_\ell(\R)$ be the matrix of the $\xi_{i,j}$, and $M_A$ be the block matrix:
\begin{equation}\label{def_MA_matricedeA_armoeigaoinv}
M_A=\begin{pmatrix} I_\ell \\ M_\xi \end{pmatrix}\in\M_{2\ell,\ell}(\R).
\end{equation}
Let us denote by $Y_1,\ldots,Y_{\ell}\in\R^{2\ell}$ the columns of $M_A$, and let $A$ be the subspace of $\R^{2\ell}$ spanned by the $Y_i$: $A=\span(Y_1,\ldots,Y_\ell)$. Notice that $\rk(M_A)=\ell$, so $\dim A=\ell$.

Let us establish that there exist sequences $(e^{(i,j)}_k)_{k\in\N}$ with values in $\{1,2\}$ if $i\ne j$ and with values in $\{2\ell,2\ell+1\}$ if $i=j$, such that $A\in\mathfrak I_{n}(\ell,\ell)_\ell$.

For better clarity, let us reindex the $\xi_{i,j}$ for $(i,j)\in\{1,\ldots,\ell\}^2$ as $\xi_1,\ldots,\xi_{\ell^2}$ by lexicographic order; the sequences $(e_k^{(i,j)})_{k\in\N}$ for $(i,j)\in\{1,\ldots,\ell\}^2$ are also reindexed as $(e_k^{(1)})_{k\in\N},\ldots,(e_k^{(\ell^2)})_{k\in\N}$ in the same way.

Let us prove by induction on $t\in\{1,\ldots,\ell^2\}$ that the sequences $(e_k^{(1)})_{k\in\N},\ldots,(e_k^{(t)})_{k\in\N}$ can be chosen such that $\xi_1,\ldots,\xi_{t}$ are $\Q$-algebraically independent. The irrationality exponent of $\xi_1$ is at least $\alpha>2$ (it is even equal to $\alpha$, see \cite{levesley06}), so with Roth's theorem (see \cite{roth55}), $\xi_1$ is transcendental. Let $t\in\{1,\ldots,\ell^2-1\}$ and let us assume that the numbers $\xi_1,\ldots,\xi_t$ are $\Q$-algebraically independent. The set of real numbers algebraic on $\Q(\xi_1,\ldots,\xi_t)$ is countable, whereas the set of the sequences $(e_k^{(t+1)})_{k\in\N}$ is not. Therefore, one can choose a sequence $(e_k^{(t+1)})_{k\in\N}$ with values in $\{1,2\}$ or $\{2\ell,2\ell+1\}$ (depending on if $t$ corresponds to a couple $(i,j)$ with $i\ne j$ or $i=j$), such that $\xi_1,\ldots,\xi_{t+1}$ are $\Q$-algebraically independent, which concludes the induction. 

Let us assume that there exists $B\in\mathfrak R_{2\ell}(\ell)$ such that $A\cap B\ne\{0\}$. Let $M_B$ be a matrix whose columns form a rational basis of $B$. Notice that $\det\begin{pmatrix} M_A & M_B\end{pmatrix}=0$ where $M_A$ is the matrix defined in Equation \eqref{def_MA_matricedeA_armoeigaoinv}. Since $M_B\in\M_{2\ell,\ell}(\Q)$, one can compute this determinant using a Laplace expansion on its $\ell$ first columns to obtain a polynomial $P\in\Q[X_1,\ldots,X_{\ell^2}]$ such that $\det\begin{pmatrix} M_A & M_B\end{pmatrix}=P(\xi_1,\ldots,\xi_{\ell^2})=0$. The fact that $\xi_1,\ldots,\xi_{\ell^2}$ are $\Q$-algebraically independent yields $P=0$. Let us decompose the matrix $M_B$ under the form $M_B=\begin{pmatrix} B_1\\ B_2\end{pmatrix}$ with $B_1,B_2\in\M_{\ell,\ell}(\R)$. The equality $P=0$ implies that
\[\forall Q\in\M_\ell(\R),\quad \Delta_Q=\det\begin{pmatrix} I_\ell & B_1 \\ Q & B_2\end{pmatrix}=0.\]
Let us mention this known claim to compute determinants of $2\times 2$ block matrices (see \cite{silvester00}, Theorem 3).
\begin{claim}\label{lemme_det_par_blocs_moaiboevodizvb}
Let $A_1,A_2,A_3,A_4\in\M_\ell(\R)$ such that $A_1A_2=A_2A_1$. Then 
\[\det\begin{pmatrix} A_1 & A_2\\A_3&A_4\end{pmatrix}=\det(A_4A_1-A_3A_2).\]
\end{claim}
%\begin{proof}[Proof of Claim \ref{lemme_det_par_blocs_moaiboevodizvb}]
%Assume that $A_1$ is invertible, and notice that
%\[\begin{pmatrix} A_1 & A_2\\A_3&A_4\end{pmatrix}\begin{pmatrix} I_\ell&-A_2A_1^{-1}\\0&I_\ell\end{pmatrix}=\begin{pmatrix} A_1&0\\A_3&-A_3A_2A_1^{-1}+A_4\end{pmatrix}\]
%because $A_1$ and $A_2$ commute, so
%\begin{equation}\label{equa_lemme_matrices_carres_abmoifnbaoinv}
%\det\begin{pmatrix} A_1 & A_2\\A_3&A_4\end{pmatrix}=\det(-A_3A_2A_1^{-1}+A_4)\det(A_1)=\det(A_4A_1-A_3A_2).
%\end{equation}
%If $A_1$ is no longer assumed invertible, the density of $\GL_{\ell}(\R)$ in $\M_\ell(\R)$ gives us a sequence of matrices $(\tilde A_N)_{N\in\N}$ of $\GL_\ell(\R)$ such that $\tilde A_N\xrightarrow[N\to+\infty]{}A_1$. Since the determinant map is continuous, Equality \eqref{equa_lemme_matrices_carres_abmoifnbaoinv} is still true.
%\end{proof}
Since $I_\ell$ commutes with $B_1$, Claim \ref{lemme_det_par_blocs_moaiboevodizvb} can be used to get
\begin{equation}\label{Delta_Q_apibpinfvpinpvinedc}
\forall Q\in\M_\ell(\R),\quad \Delta_Q=\det(B_2-QB_1)=0.
\end{equation}
Let $\lambda\in\R$; with $Q=\lambda I_\ell$, one has $\det(B_2-\lambda B_1)=0$. Assume that $B_1$ is invertible, then
\[0=\Delta_Q=\det((B_2B_1^{-1}-\lambda I_\ell)B_1)=\det(B_2B_1^{-1}-\lambda I_\ell)\det(B_1).\]
Thus, the fact that $\det(B_1)\ne0$ yields that for all $\lambda\in\R$, $\det(B_2B_1^{-1}-\lambda I_\ell)=0$. Therefore, for all $\lambda\in\R$, $\lambda$ is an eigenvalue of $B_2B_1^{-1}$, and this can not be, so $\det(B_1)=0$. Let $r=\rk(B_1)<\ell$, let $U,V\in\GL_\ell(\R)$ be two invertible matrices such that
\[UB_1V=\begin{pmatrix} I_r & 0 \\ 0 & 0\end{pmatrix}=\begin{pmatrix} J_r& 0\end{pmatrix}\in\M_\ell(\R)\quad\text{ where }\quad J_r=\begin{pmatrix} I_r\\0\end{pmatrix}\in\M_{\ell,r}(\R).\]
Let us decompose $UB_2V$ as $UB_2V=\begin{pmatrix} C_1& C_2\end{pmatrix}\in\M_\ell(\R)$ where the matrix $C_1\in\M_{\ell,r}(\R)$ is formed with the $r$ first columns of $UB_2V$, and the matrix $C_2\in\M_{\ell,\ell-r}(\R)$ is formed with the $\ell-r$ last columns of $UB_2V$. Thus, the matrices $\begin{pmatrix} J_r & 0 \\ C_1 & C_2 \end{pmatrix}$ and $\begin{pmatrix} B_1\\B_2\end{pmatrix}$ are equivalent since
\begin{equation}\label{equivdeuxmatrices_moareinmoaunvomfn}
\begin{pmatrix} J_r & 0 \\ C_1 & C_2 \end{pmatrix}=\begin{pmatrix} U & 0 \\ 0 & U\end{pmatrix}
\begin{pmatrix} B_1\\B_2\end{pmatrix}V\in\M_{2\ell,\ell}(\R).
\end{equation}
Since $I_\ell$ commutes with $UB_1V$, Claim \ref{lemme_det_par_blocs_moaiboevodizvb} implies that for all $Q\in\M_\ell(\R)$:
\[\det\begin{pmatrix} I_\ell & UB_1V \\ UQU^{-1} & UB_2V\end{pmatrix}=\det(UB_2V-UQU^{-1}UB_1V)=\det(U)\Delta_Q\det (V)=0\]
using Equation \eqref{Delta_Q_apibpinfvpinpvinedc}. Since this is true for all $Q\in\M_\ell(\R)$, let $Q'=UQU^{-1}$ to get
\[\forall Q'\in\M_\ell(\R),\quad \Delta'_{Q'}=\det\begin{pmatrix} I_\ell & UB_1V \\ Q' & UB_2V\end{pmatrix}=0.\]
Let $R\in\M_{\ell,r}(\R)$, and let us define a block matrix as $Q'=\begin{pmatrix} C_1-R & 0\end{pmatrix}\in\M_\ell(\R)$. Since $I_\ell$ et $UB_1V$ commute, Claim \ref{lemme_det_par_blocs_moaiboevodizvb} implies that
\[0=\Delta'_{Q'}=\det(UB_2V-Q'UB_1V)=\det\left(\begin{pmatrix} C_1 & C_2\end{pmatrix}-\begin{pmatrix} C_1-R & 0\end{pmatrix}\begin{pmatrix} I_r & 0\\0&0\end{pmatrix}\right)=\det\begin{pmatrix} R & C_2\end{pmatrix}.\]
If $\rk(C_2)=\ell-r$, then it would be possible to find $R$ such that $\rk\begin{pmatrix} R & C_2\end{pmatrix}=\ell$, which can not be since its determinant $\Delta'_{Q'}$ would be non-zero. Therefore, $\rk(C_2)<\ell-r$. Equation \eqref{equivdeuxmatrices_moareinmoaunvomfn} yields
\[\rk(M_B)=\rk\begin{pmatrix} B_1 \\ B_2\end{pmatrix}=\rk\begin{pmatrix} J_r & 0 \\ C_1 & C_2 \end{pmatrix}=r+\rk(C_2)<r+\ell-r=\ell,\]
which can not be since $\dim B=\ell=\rk(M_B)$; hence $A\cap B=\{0\}$ for all $B\in\mathfrak R_{2\ell}(\ell)$, \emph{i.e.} $A\in\mathfrak I_{2\ell}(\ell,\ell)_1\subset\mathfrak I_{n}(\ell,\ell)_\ell$. \\

The subspace $A$ having been constructed, let us construct rational subspaces $B_N$ for $N\ge 1$ approaching $A$ to its $\ell$-th angle to the exponent exactly $\beta$. Then, we will show that these subspaces $B_N$ are the ones providing the best approximation of $A$ to its $\ell$-th angle, which will finally give $\muexpA nA\ell\ell=\beta$. 

For $(i,j)\in\{1,\ldots,\ell\}^2$ and $N\ge 1$, let
\[f_N^{(i,j)}=\theta^{\lfloor\alpha^N\rfloor}\sum_{k=0}^N \frac{e^{(i,j)}_k}{\theta^{\lfloor\alpha^k\rfloor}}\in\Z\]
and $M_{B_N}$ be the block matrix
\[M_{B_N}=\begin{pmatrix} \theta^{\lfloor\alpha^N\rfloor}I_\ell \\ F_N\end{pmatrix}\in\M_{2\ell,\ell}(\Z)\]
where $F_N$ is the matrix $(f_N^{(i,j)})_{(i,j)\in\{1,\ldots,\ell\}^2}\in\M_\ell(\Z)$. Let us denote by $X_N^{(1)},\ldots,X_N^{(\ell)}$ the columns of $M_{B_N}$, and let
\[B_N=\span(X_N^{(1)},\ldots,X_N^{(\ell)})\in\mathfrak R_{2\ell}(\ell).\]
One can notice that for all $(i,j)\in\{1,\ldots,\ell\}^2$:
\[\sum_{k=N+1}^\infty \frac{e_k^{(i,j)}}{\theta^{\lfloor\alpha^k\rfloor}}\le (2\ell+1)\sum_{j=\lfloor\alpha^{N+1}\rfloor}^\infty \frac 1{\theta^j}< \frac{4\ell+2}{\theta^{\lfloor\alpha^{N+1}\rfloor}}\]
because $\theta> 2$, so
\begin{equation}\label{le_sev_approche_bien_moaeirgnomdfvn}
0<\xi_{i,j}-\frac{f_N^{(i,j)}}{\theta^{\lfloor\alpha^N\rfloor}}=\sum_{k=N+1}^\infty \frac{e_k^{(i,j)}}{\theta^{\lfloor\alpha^k\rfloor}}<\frac {4\ell+2}{\theta^{\lfloor\alpha^{N+1}\rfloor}}.
\end{equation}
Let us show that there exists a constant $c_1>0$ depending only on $Y_1,\ldots,Y_\ell$ such that
\[\forall N\ge 1,\quad \psi_\ell(A,B_N)\le \frac{c_1}{H(B_N)^{\alpha/\ell}},\]
which will imply that $\muexpA {2\ell}A{\ell}{\ell}\ge \alpha/\ell$. In order to do this, let us establish an upper bound for the height of the $B_N$. We will see later (Claim \ref{lemmeminorationhauteurdesBN_aeoinmoaeivnv}) that this upper bound is in fact optimal up to a multiplicative constant.
\begin{claim}\label{majoration_hauteur_B_N_painoaeoefainveo}
For all $N\ge 1$, one has
\[H(B_N)\le c_2(\theta^{\lfloor\alpha^N\rfloor})^\ell\]
where $c_2>0$ depends only on $\ell$.
\end{claim}
\begin{proof}
Because $\theta\ge 2$, one has
\begin{equation}\label{maj_f_N_ij_aouomfaeuboueabv}
\abs{f_N^{(i,j)}}\le (2\ell+1)\theta^{\lfloor\alpha^N\rfloor}\sum_{k=0}^N\frac 1{\theta^{\lfloor\alpha^k\rfloor}}\le 2(2\ell+1)\theta^{\lfloor\alpha^N\rfloor}.
\end{equation}
Therefore, because all the $2\ell$ coefficients of each $X_N^{(j)}$ are smaller than $2(2\ell+1)\cdot \theta^{\lfloor\alpha^N\rfloor}$:
\[H(B_N)\le \norme{X_N^{(1)}\wedge\cdots\wedge X_N^{(\ell)}}\le \prod_{j=1}^\ell\norme{X_N^{(j)}}\le (2(2\ell+1)\cdot \sqrt{2\ell})^\ell (\theta^{\lfloor\alpha^N\rfloor})^\ell.\]
\end{proof}
Let us state a special case of Lemma 6.1 of \cite{joseph21bis} which will be used below.
\begin{lemma}\label{resultat_reconstruction_proximite_oaeirgbvodbv}
Let $F_1,\ldots,F_\ell, B_1,\ldots,B_\ell$ be $2\ell$ lines of $\R^{2\ell}$. Assume that the $F_i$ span a subspace of dimension $\ell$ and so do the $B_i$. Let $F=F_1\oplus\cdots\oplus F_\ell$ and $B=B_1\oplus\cdots\oplus B_\ell$, then one has
\[\psi_\ell(F,B)\le c_{F}\sum_{i=1}^\ell\psi_{1}(F_i,B_i)\]
where $c_{F}>0$ is a constant depending only on $F_1,\ldots,F_\ell$.
\end{lemma}
For $i\in\{1,\ldots,\ell\}$, let $Z_N^{(i)}=\theta^{-\lfloor\alpha^N\rfloor}X_N^{(i)}$. Notice that the definition of $\psi(X,Y)$ leads to the following elementary claim.
\begin{claim}\label{claim_trigo_majoration_oaeibnoenvn}
If $X$ and $Y$ are non-zero vectors, then $\psi(X,Y)\le \norme{X-Y}/\norme{X}$.
\end{claim}
Here, $\norme{Y_i}\ge 1$, so Claim \ref{claim_trigo_majoration_oaeibnoenvn} combined with Inequality \eqref{le_sev_approche_bien_moaeirgnomdfvn} implies that
\begin{equation}\label{majoration_psi_XNYi_zmofihgmenv} 
\psi(X_N^{(i)},Y_i)=\psi(Z_N^{(i)},Y_i) \le \frac{\norme{Z_N^{(i)}-Y_i}}{\norme{Y_i}}\le c_3(\theta^{\lfloor\alpha^N\rfloor})^{-\alpha}
\end{equation}
with $c_3>0$ depending only on $\ell$. Lemma \ref{resultat_reconstruction_proximite_oaeirgbvodbv} gives a constant $c_4>0$ depending only on $Y_1,\ldots,Y_\ell$ such that 
\begin{equation}\label{maj_psi_ell_A_B_N_aoeribqfoisbcovbz}
\psi_\ell(A,B_N)\le c_4\sum_{i=1}^\ell \psi(X_N^{(i)},Y_i)\le \frac{c_5}{(\theta^{\lfloor\alpha^N\rfloor})^{\alpha}}
\end{equation}
with $c_5>0$ depending only on $Y_1,\ldots,Y_\ell$. Using Claim \ref{majoration_hauteur_B_N_painoaeoefainveo} which assets that $H(B_N)\le c_2(\theta^{\lfloor\alpha^N\rfloor})^\ell$, this yields $\psi_\ell(A,B_N) \le c_1 H(B_N)^{-\alpha/\ell}$. \\

Now, we will show that the $B_N$ achieve the best approximation of $A$ to its $\ell$-th angle. Concretely, let us prove that if $\epsilon>0$ and $C\in\mathfrak R_{2\ell}(\ell)$ are such that 
\begin{equation}\label{hyp_sur_psi2_A_C}
\psi_\ell(A,C)\le \frac 1{H(C)^{\alpha/\ell+\epsilon}}
\end{equation}
and if $H(C)$ is large enough (in terms of $\ell$ and $\epsilon$), then there exists $N\ge 1$ such that $C=B_N$.

Since $C$ is a rational subspace, there exist $v_1,\ldots,v_\ell\in\Z^{2\ell}$ such that $(v_1,\ldots,v_\ell)$ is a $\Z$-basis of $C\cap\Z^{2\ell}$. One has $H(C)=\norme{v_1\wedge\cdots\wedge v_\ell}$ with Theorem \ref{th_def_equiv_hauteur_vaoribibgipn}. To prove that $C=B_N$ for some integer $N\ge 1$, let us show that all the $X_N^{(i)}$ for $i\in\{1,\ldots,\ell\}$ are in $C=\span(v_1,\ldots,v_\ell)$. Since $\dim C=\dim B_N$, it will imply that $C=B_N$. Let $N\ge 1$ and $i\in\{1,\ldots,\ell\}$; let us consider the $\ell+1$ vectors $X_N^{(i)},v_1,\ldots,v_\ell$, and let $Q=\begin{pmatrix} X_N^{(i)}&v_1&\cdots&v_\ell\end{pmatrix}\in\M_{2\ell,\ell+1}(\Z)$. Since $v_1,\ldots,v_\ell$ are linearly independent, to show that $X_N^{(i)}\in\span(v_1,\ldots,v_\ell)$, it is sufficient to show that $\rk(Q)<\ell+1$, \emph{i.e.} that all $(\ell+1)\times(\ell+1)$ minors of $Q$ are zero. For this purpose, let us establish that $D=\norme{X_N^{(i)}\wedge v_1\wedge\cdots\wedge v_\ell}=0$.

Let $p_C^\perp$ be the orthogonal projection onto $C$ and $h$ the vector $h=p_C^\perp(X_N^{(i)})-X_N^{(i)}$. There exist $\lambda_1,\ldots,\lambda_\ell\in\R$ such that $X_N^{(i)}$ can be written as $X_N^{(i)}=\sum_{j=1}^\ell\lambda_jv_j-h$. Because $h\in C^\perp$, one has
\[D=\norme{\left(\sum_{j=1}^\ell\lambda_jv_j-h\right)\wedge v_1\wedge\cdots\wedge v_\ell}=\norme h \cdot \norme{v_1\wedge\cdots\wedge v_\ell}=\norme h H(C).\]
Moreover, $\norme h=\norme{X_N^{(i)}}\psi_1(X_N^{(i)},C)$, so
\[D\le c_6 \theta^{\lfloor\alpha^N\rfloor} (\psi(X_N^{(i)},Y_i)+\psi_1(\span(Y_i),C))H(C)\] 
with $c_6>0$ depending only on $\ell$, using Equation \eqref{maj_f_N_ij_aouomfaeuboueabv} and a triangle inequality on $\psi$ (for all non-zero vectors $Z_1,Z_2,Z_3$, one has $\psi(Z_1,Z_2)\le\psi(Z_1,Z_3)+\psi(Z_3,Z_2)$; see \cite{schmidt67}, Equation (3) page 446). Lemma \ref{lemme_transfert_psi1_psiell_baoeoearv} yields $\psi_1(\span(Y_i),C)\le \psi_\ell(A,C)$; using the upper bound \eqref{majoration_psi_XNYi_zmofihgmenv} to deal with $\psi(Y_i,X_N^{(i)})$ and the upper bound \eqref{hyp_sur_psi2_A_C} to deal with $\psi_\ell(A,C)$, one gets
\begin{equation}\label{majoration_de_D_aobvobobdvipsn}
D\le c_6 \theta^{\lfloor\alpha^N\rfloor}H(C)\left(\frac {c_3}{(\theta^{\lfloor\alpha^N\rfloor})^{\alpha}}+\frac 1{H(C)^{\alpha/\ell+\epsilon}}\right)\le c_7 \left(\frac {H(C)}{\theta^{\lfloor\alpha^N\rfloor(\alpha-1)}}+\frac{\theta^{\lfloor\alpha^N\rfloor}}{H(C)^{\alpha/\ell-1+\epsilon}}\right)
\end{equation}
with $c_7>0$ depending only on $\ell$.

From now on, let us choose a particular $N$: let $N$ be the largest integer such that $\theta^{\alpha^N}\le H(C)^{\alpha/\ell-1+\epsilon/2}$. Notice that $\theta^{\lfloor\alpha^N\rfloor}\le H(C)^{\alpha/\ell-1+\epsilon/2}$. Because $N$ is maximal, one has $(\theta^{\alpha^N})^\alpha=\theta^{\alpha^{N+1}}> H(C)^{\alpha/\ell-1+\epsilon/2}$, so $\theta^{\alpha^N}>H(C)^{(\alpha/\ell-1+\epsilon/2)/\alpha}$. Since $\alpha=\ell\beta\ge(2\ell+1+\sqrt{1+4\ell^2})/2$ yields $\alpha^2-(2\ell+1)\alpha+\ell\ge0$, so $(\alpha/\ell-1)/\alpha\ge 1/(\alpha-1)$. Thus, $\theta^{\alpha^N}>H(C)^{1/(\alpha-1)+\epsilon/(2\alpha)}$, whence $
\theta^{\lfloor\alpha^N\rfloor}>\theta^{-1} H(C)^{1/(\alpha-1)+\epsilon/(2\alpha)}$. Therefore, coming back to Inequality \eqref{majoration_de_D_aobvobobdvipsn} gives
\[
D\le c_8\left(\frac 1{H(C)^{(\alpha-1)\epsilon/(2\alpha)}}+\frac 1{H(C)^{\epsilon/2}}\right)\xrightarrow[H(C)\to+\infty]{} 0\]
with $c_8>0$ depending only on $\ell$. If $H(C)$ is large enough (in terms of $\ell$ and $\epsilon$, because $(\alpha-1)/(2\alpha)\ge(\ell-1)/(2\ell)$), one has $D<1$. But $E=\norme{X_N^{(i)}\wedge v_1\cdots\wedge v_\ell}_\infty$ is a positive integer such that $E\le D$, so $E=0$, \emph{i.e.} $X_N^{(i)}\wedge v_1\wedge\cdots\wedge v_\ell=0$. Thus, it has been shown that if $H(C)$ is large enough, $C=B_N$ where $N$ is the largest integer such that $\theta^{\alpha^N}\le H(C)^{\alpha/\ell-1+\epsilon/2}$. \\

Now that it has been established that if $N\in\N^*$ is large enough, the rational subspaces $B_N$ give the best possible approximation of $A$ to its $\ell$-th angle, let us show that they approach the subspace $A$ at most to the exponent $\alpha/\ell$. In other words, we shall prove that for $N$ large enough, one has
\[\psi_\ell(A,B_N)\ge \frac{c}{H(B_N)^{\alpha/\ell}}\]
with $c>0$ depending only on $Y_1,\ldots,Y_\ell$. For this, we need to establish a lower bound on the height of the $B_N$ (which will imply that the upper bound in Claim \ref{majoration_hauteur_B_N_painoaeoefainveo} is optimal up to a multiplicative constant).
\begin{claim}\label{lemmeminorationhauteurdesBN_aeoinmoaeivnv}
For all $N$ large enough, one has
\[H(B_N)\ge \tilde c(\theta^{\lfloor\alpha^N\rfloor})^\ell\]
with $\tilde c>0$ depending only on $Y_1,\ldots,Y_\ell$.
\end{claim}
\begin{proof}[Proof of Claim \ref{lemmeminorationhauteurdesBN_aeoinmoaeivnv}.]
Let $N\ge 1$; let us establish that the family $(X_N^{(1)},\ldots,X_N^{(\ell)})$ is a $\Z$-basis of $B_N\cap\Z^{2\ell}$. For this purpose, let us denote by $P$ the parallelotope spanned by $X_N^{(1)},\ldots,X_N^{(\ell)}$, \emph{i.e.}
\[P=\left\{\sum_{i=1}^\ell \lambda_iX_N^{(i)},\ (\lambda_1,\ldots,\lambda_\ell)\in[0,1]^\ell\right\},\]
and let us show that the $2^\ell$ vertices of $P$ are its only integer points. Let $\mathcal S$ be the set of the $2^\ell$ vertices of $P$, \emph{i.e.}
\[\mathcal S=\left\{\sum_{i=1}^\ell \delta_iX_N^{(i)},\ (\delta_1,\ldots,\delta_\ell)\in\{0,1\}^\ell\right\}.\]
Assume that there exists $X\in (P\setminus\mathcal S)\cap\Z^{2\ell}$, and let $(\lambda_1,\ldots,\lambda_\ell)\in[0,1]^\ell\setminus\{0,1\}^\ell$ be such that $X=\lambda_1 X_N^{(1)}+\cdots+\lambda_\ell X_N^{(\ell)}\in\Z^{2\ell}$. The first $\ell$ coordinates of $X$ give that for all $i\in\{1,\ldots,\ell\}$, $\lambda_i \theta^{\lfloor\alpha^N\rfloor}\in\Z$. Thus, there exist integers $\gamma_1,\ldots,\gamma_\ell\in\{0,\ldots,\theta^{\lfloor\alpha^N\rfloor}\}$ such that for all $i\in\{1,\ldots,\ell\}$, $ \lambda_i=\gamma_i\theta^{-\lfloor\alpha^N\rfloor}$, because the $\lambda_i$ are in $[0,1]$. Moreover, the last $\ell$ coordinates of $X$ give that for all $i\in\{1,\ldots,\ell\}$, $\lambda_1 f_N^{(i,1)}+\cdots+\lambda_\ell f_N^{(i,\ell)}\in\Z$, so
\begin{equation}\label{random_riozagoeigbzouefm}\forall i\in\{1,\ldots,\ell\},\quad \sum_{j=1}^\ell \frac{\gamma_j}{\theta^{\lfloor\alpha^N\rfloor}}\cdot\theta^{\lfloor\alpha^N\rfloor} \sum_{k=0}^N\frac{e_k^{(i,j)}}{\theta^{\lfloor\alpha^k\rfloor}}=\sum_{k=0}^N\frac 1{\theta^{\lfloor\alpha^k\rfloor}}\sum_{j=1}^\ell\gamma_j e_k^{(i,j)}\in\Z.
\end{equation}
For $k\in\{0,\ldots,N\}$, let us denote by $E_k$ the matrix $(e_k^{(i,j)})_{(i,j)\in\{1,\ldots,\ell\}^2}\in\M_{\ell}(\Z)$, and by $\Gamma$ the column vector $\transp(\gamma_1,\ldots,\gamma_\ell)$. Thus, the $\ell$ equations given by \eqref{random_riozagoeigbzouefm} can be rewritten using matrices as
\[\begin{pmatrix} \displaystyle\sum_{k=0}^N \frac{e_k^{(1,1)}}{\theta^{\lfloor\alpha^k\rfloor}} & \cdots & \displaystyle\sum_{k=0}^N \frac{e_k^{(1,\ell)}}{\theta^{\lfloor\alpha^k\rfloor}} \\ \vdots & & \vdots \\ \displaystyle\sum_{k=0}^N \frac{e_k^{(\ell,1)}}{\theta^{\lfloor\alpha^k\rfloor}} & \cdots & \displaystyle\sum_{k=0}^N \frac{e_k^{(\ell,\ell)}}{\theta^{\lfloor\alpha^k\rfloor}}\end{pmatrix}\begin{pmatrix} \gamma_1\\\vdots\\\vdots\\\gamma_\ell\end{pmatrix}\in\Z^\ell,\]
which becomes
\[\sum_{k=0}^N  \frac{1}{\theta^{\lfloor\alpha^k\rfloor}} E_k\Gamma\in\Z^\ell,\]
and to highlight the last term of this sum:
\[\sum_{k=0}^{N-1}  \theta^{\lfloor\alpha^N\rfloor-\lfloor\alpha^k\rfloor} E_k\Gamma+E_N\Gamma\in\theta^{\lfloor\alpha^N\rfloor}\Z^\ell.\]
Since $E_N\in\M_\ell(\Z)$, the transpose of its comatrix also belongs to $\M_\ell(\Z)$, so
\[\sum_{k=0}^{N-1}  \theta^{\lfloor\alpha^N\rfloor-\lfloor\alpha^k\rfloor} \transp\com(E_N)E_k\Gamma+\det(E_N)\Gamma\in\theta^{\lfloor\alpha^N\rfloor}\Z^\ell.\]
For $i\in\{1,\ldots\ell\}$ and $k\in\{0,\ldots,N-1\}$, let us denote by $L_{k,i}\in\M_{1,\ell}(\Z)$ the $i$-th row of the product $\transp\com(E_N)E_k$. Thus,
\begin{equation}\label{random_omerganmoernbme}
\forall i\in\{1,\ldots,\ell\},\quad \sum_{k=0}^{N-1}  (L_{k,i}\Gamma)\theta^{\lfloor\alpha^N\rfloor-\lfloor\alpha^k\rfloor}+\det(E_N)\gamma_i\in\theta^{\lfloor\alpha^N\rfloor}\Z.
\end{equation}
Let $j\in\{1,\ldots,\ell\}$; notice that $e_N^{(j,j)}\ge2\ell>\sum_{i\ne j} e_N^{(i,j)}$, so $E_N$ is a strictly diagonally dominant matrix (the case $\ell=1$ being trivial). Therefore $E_N$ is invertible, so $\det(E_N)\ne 0$. Moreover, $\abs{e_N^{(i,j)}}\le 2$ if $i\ne j$ and $\abs{e_N^{(i,i)}}\le 2\ell+1$, so by definition of $\theta$ (see Inequality \eqref{hyp_sur_theta_amorinvamnoavcbo}):
\[\abs{\det(E_N)}=\abs{\sum_{\sigma\in\mathfrak S_\ell}\epsilon(\sigma)\prod_{i=1}^\ell e_N^{(i,\sigma(i))}} \le\ell!\,(2\ell+1)^\ell <\theta.\]
Thus, since $0<\abs{\det(E_N)}<\theta$, one has $v_\theta(\det(E_N))=0$, so $v_\theta(\det(E_N)\gamma_i)=v_\theta(\gamma_i)$. Let $u\ge 0$ and $i_0\in\{1,\ldots,\ell\}$ be such that $u=\min(v_\theta(\gamma_1),\ldots,v_\theta(\gamma_\ell))=v_\theta(\gamma_{i_0})$. Coming back to Equation \eqref{random_omerganmoernbme} yields
\[\forall i\in\{1,\ldots,\ell\},\quad v_\theta\left(\sum_{k=0}^{N-1}  (L_{k,i}\Gamma)\theta^{\lfloor\alpha^N\rfloor-\lfloor\alpha^k\rfloor}+\det(E_N)\gamma_i\right)\ge\lfloor\alpha^N\rfloor,\]
with the convention $v_\theta(0)=+\infty$. For all $k\in\{0,\ldots,N-1\}$, $v_\theta(\theta^{\lfloor\alpha^N\rfloor-\lfloor\alpha^k\rfloor})\ge \lfloor\alpha^N\rfloor-\lfloor\alpha^{N-1}\rfloor>0$, and since $L_{k,i}\Gamma$ is a $\Z$-linear combination of the $\gamma_i$, $v_\theta(L_{k,i}\Gamma)\ge \min(v_\theta(\gamma_1),\ldots,v_\theta(\gamma_\ell))=u$. The particular case $i=i_0$ yields
\[v_\theta\left(\sum_{k=0}^{N-1}  (L_{k,i_0}\Gamma)\theta^{\lfloor\alpha^N\rfloor-\lfloor\alpha^k\rfloor}+\det(E_N)\gamma_{i_0}\right)=u\ge\lfloor\alpha^N\rfloor.\]
Whence, by definition of $u$, for all $i\in\{1,\ldots,\ell\}$, one has $v_\theta(\gamma_i)\ge \lfloor\alpha^N\rfloor$. Since all the $\gamma_i$ are in $\{0,\ldots,\theta^{\lfloor\alpha^N\rfloor}\}$, this implies that for all $i\in\{1,\ldots,\ell\}$, $\gamma_i\in\big\{0,\theta^{\lfloor\alpha^N\rfloor}\big\}$. Thus, $X\in\mathcal S$, which can not be.

It has been shown that the only integers points of $P$ are the ones in $\mathcal S$, which implies that the family $(X_N^{(1)},\ldots,X_N^{(\ell)})$ is a $\Z$-basis of $B_N\cap\Z^{2\ell}$. Thus, using Theorem \ref{th_def_equiv_hauteur_vaoribibgipn}, $H(B_N)=\norme{X_N^{(1)}\wedge\cdots\wedge X_N^{(\ell)}}$. But
\[\norme{\theta^{-\lfloor\alpha^N\rfloor} X_N^{(1)}\wedge\cdots\wedge\theta^{-\lfloor\alpha^N\rfloor}X_N^{(\ell)}}\xrightarrow[N\to+\infty]{}\norme{Y_1\wedge\cdots\wedge Y_\ell},\]
so for $N$ large enough:
\[H(B_N)=\left(\theta^{\lfloor\alpha^N\rfloor}\right)^\ell\norme{\theta^{-\lfloor\alpha^N\rfloor} X_N^{(1)}\wedge\cdots\wedge\theta^{-\lfloor\alpha^N\rfloor}X_N^{(\ell)}}\ge \tilde c\left(\theta^{\lfloor\alpha^N\rfloor}\right)^\ell\]
with $\tilde c>0$ depending only on $Y_1,\ldots,Y_\ell$.
\end{proof}
Let $Z_N^{(1)}=\theta^{-\lfloor\alpha^N\rfloor}X_N^{(1)}$ and $p_A^\perp$ be the orthogonal projection onto $A$. Lemma \ref{lemme_transfert_psi1_psiell_baoeoearv} gives
\begin{equation}\label{minoration_psiell_AB_N_aroighmeobvo}
\psi_\ell(A,B_N)\ge \psi_1(\span(Z_N^{(1)}),A)= \psi(Z_N^{(1)},p_A^\perp(Z_N^{(1)})).
\end{equation}
Let $\Delta=p_A^\perp(Z_N^{(1)})-Y_1$ and $\omega=\norme{p_A^\perp(Z_N^{(1)})-Z_N^{(1)}}$. Let us decompose the vector $p_A^\perp(Z_N^{(1)})$ in the basis $(Y_1,\ldots,Y_\ell)$:
\[p_A^\perp(Z_N^{(1)})=\sum_{i=1}^\ell \lambda_i Y_i=\transp \begin{pmatrix}\lambda_1&\cdots&\lambda_\ell&\star&\cdots&\star\end{pmatrix}\]
where the $\star$ are unspecified coefficients, because for all $i\in\{1,\ldots,\ell\}$, the vector $Y_i$ can be written $Y_i=\transp \begin{pmatrix} \delta_{i,1}&\cdots&\delta_{i,\ell}&\star&\cdots&\star\end{pmatrix}$ where $\delta$ is the Kronecker delta. Moreover, for all $i\in\{1,\ldots,\ell\}$, $Z_N^{(i)}=\transp \begin{pmatrix} \delta_{i,1}&\cdots&\delta_{i,\ell}&\star&\cdots&\star\end{pmatrix}$, so $\omega^2=\norme{p_A^\perp(Z_N^{(1)})-Z_N^{(1)}}^2\ge (\lambda_1-1)^2+\displaystyle\sum_{i=2}^\ell \lambda_i^2$. Thus, $\abs{\lambda_1-1}\le\omega$, and for all $i\in\{2,\ldots,\ell\}$, $\abs{\lambda_i}\le\omega$. Let $j\in\{1,\ldots,\ell\}$. One has $\norme{Y_j}^2=1+\displaystyle\sum_{i=1}^\ell\left(\sum_{k=0}^\infty \frac{e_k^{(i,j)}}{\theta^{\lfloor \alpha^k\rfloor}}\right)^2$, but $\alpha>2$, $\theta\ge 2$ and $2\ell+1\le \theta$ with Hypothesis \eqref{hyp_sur_theta_amorinvamnoavcbo} on $\theta$, so a simple computation gives $\norme{Y_j}\le \sqrt{1+4\ell}=c_9$. Notice that $\Delta=p_A^\perp(Z_N^{(1)})-Y_1=(\lambda_1-1)Y_1+\sum_{i=2}^\ell\lambda_i Y_i$, so $\norme{\Delta}\le c_{10}\omega$, with $c_{10}>0$ depending only on $\ell$. One has
\begin{equation}\label{minoration_norme_Z_N1_proj_maoeboefb}
\norme{Z_N^{(1)}\wedge p_A^\perp(Z_N^{(1)})}=\norme{Z_N^{(1)}\wedge(Y_1+p_A^\perp(Z_N^{(1)})-Y_1)}\ge \norme{Z_N^{(1)}\wedge Y_1}-\norme{Z_N^{(1)}\wedge \Delta}.
\end{equation}
Notice that $\begin{pmatrix} Z_N^{(1)} &Y_1\end{pmatrix}\in\M_{2\ell,2}(\R)$, and let us denote by $\eta_{i,j}$ its the $2\times 2$ minor corresponding to its $i$-th and $j$-th rows with $i<j$; one has
\[\norme{Z_N^{(1)}\wedge Y_1}=\sqrt{\sum_{1\le i<j\le 2\ell} \eta_{i,j}^2}\ge \abs{\eta_{1,\ell+1}}=\eta_{1,\ell+1}=\begin{vmatrix} 1 & 1 \\\displaystyle\sum_{k=0}^N\frac{e_k^{(1,1)}}{\theta^{\lfloor\alpha^k\rfloor}} &\displaystyle\sum_{k=0}^\infty\frac{e_k^{(1,1)}}{\theta^{\lfloor\alpha^k\rfloor}}\end{vmatrix}\ge \frac 1{\theta^{\lfloor\alpha^{N+1}\rfloor}}.\]
Moreover, in the same fashion that it was shown above that $\norme{Y_j}\le c_9$, one can establish that $\norme{Z_N^{(1)}}\le c_9$. Since $\norme{\Delta}\le c_{10}\omega$, it yields $\norme{Z_N^{(1)}\wedge \Delta}\le \norme{Z_N^{(1)}}\cdot\norme{\Delta}\le c_{11}\omega$ with $c_{11}>0$ depending only on $\ell$. With Inequality \eqref{minoration_norme_Z_N1_proj_maoeboefb}, $\norme{Z_N^{(1)}\wedge p_A^\perp(Z_N^{(1)})}\ge \theta^{-\lfloor\alpha^{N+1}\rfloor}-c_{11}\omega$. Because $\norme{p_A^\perp(Z_N^{(1)})}\le \norme{Z_N^{(1)}}\le c_9$, one has
\[\omega=\norme{p_A^\perp(Z_N^{(1)})-Z_N^{(1)}}=\norme{Z_N^{(1)}}\psi(p_A^\perp(Z_N^{(1)}),Z_N^{(1)})=\norme{Z_N^{(1)}}\frac{\norme{p_A^\perp(Z_N^{(1)})\wedge Z_N^{(1)}}}{\norme{Z_N^{(1)}}\cdot\norme{p_A^\perp(Z_N^{(1)})}}\ge \frac {c_{12}}{\theta^{\lfloor\alpha^{N+1}\rfloor}}-c_{13}\omega\]
with $c_{12},c_{13}>0$ depending only on $\ell$. Hence, $\omega\ge c_{12}(1+c_{13})^{-1}\theta^{-\lfloor\alpha^{N+1}\rfloor}=c_{14}\theta^{-\lfloor\alpha^{N+1}\rfloor}$ with $c_{14}>0$ depending only on $\ell$. Let us use Inequality \eqref{minoration_psiell_AB_N_aroighmeobvo} to get:
\begin{equation}\label{minoration_psiell_oaioeanvoienvoeinvoevni}
\psi_\ell(A,B_N)\ge \psi(Z_N^{(1)},p_A^\perp(Z_N^{(1)}))=\frac{\omega}{\norme{Z_N^{(1)}}}\ge \frac {c_{15}}{\theta^{\lfloor\alpha^{N+1}\rfloor}}
\end{equation}
with $c_{15}>0$ depending only on $\ell$. Notice that $\lfloor\alpha^{N+1}\rfloor\le \alpha^{N+1}\le \lfloor\alpha^N\rfloor\alpha+\alpha$, so $\theta^{-\lfloor\alpha^{N+1}\rfloor}\ge \theta^{-\lfloor\alpha^{N}\rfloor\alpha-\alpha}=\theta^{-\alpha}\cdot(\theta^{-\lfloor\alpha^{N}\rfloor})^\alpha$. Moreover, Claim \ref{lemmeminorationhauteurdesBN_aeoinmoaeivnv} gives a constant $c_{16}>0$, depending only on $Y_1,\ldots,Y_\ell$, such that $H(B_N)\ge c_{16}(\theta^{\lfloor\alpha^N\rfloor})^\ell$. Thus, with Inequality \eqref{minoration_psiell_oaioeanvoienvoeinvoevni}, one has
\[\psi_\ell(A,B_N)\ge \frac {c_{15}}{\theta^{\lfloor\alpha^{N+1}\rfloor}}\ge \frac {c_{15}}{\theta^{\alpha}} \cdot\frac 1{\theta^{\alpha\lfloor\alpha^N\rfloor}}\ge \frac{c_{17}}{H(B_N)^{\alpha/\ell}}\]
with $c_{17}>0$ depending only on $Y_1,\ldots,Y_\ell$. 

Notice that we have just proved: $\muexpA {2\ell}{A}{\ell}\ell\le \beta$; therefore $A$ is such that $\muexpA {2\ell}A\ell\ell=\beta$. \\

Finally, only the cases $\beta=+\infty$ and $n>2\ell$ remain to prove. Let us start by assuming that $n=2\ell$. If $\beta=+\infty$, for $(i,j)\in\{1,\ldots,\ell\}^2$ let $\xi_{i,j}=\sum_{k=0}^\infty e_k^{(i,j)} 3^{-k^k}$ where the $(e_k^{(i,j)})_{k\in\N}$ are sequences yet to be determined, with values in $\{1,2\}$. With the same notations as before, let $M_\xi=(\xi_{i,j})_{(i,j)\in\{1,\ldots,\ell\}^2}\in\M_\ell(\R)$ and let us denote by $A_\infty$ the subspace spanned by the columns $Y_1,\ldots,Y_\ell$ of the matrix $\begin{pmatrix} I_\ell \\ M_\xi\end{pmatrix}\in\M_{2\ell,\ell}(\R)$. In the same way as it was done above, one can choose sequences $(e_k^{(i,j)})_{k\in\N}$, for $(i,j)\in\{1,\ldots,\ell\}^2$, such that $A_\infty\in\mathfrak I_n(\ell,\ell)_\ell$. For $(i,j)\in\{1,\ldots,\ell\}^2$ and $N\ge 1$, let $f_N^{(i,j)}=3^{N^N}\sum_{k=0}^N e_k^{(i,j)} 3^{-k^k}$, and let us denote by $B_N\in\mathfrak R_{2\ell}(\ell)$ the rational subspace spanned by the columns of $\begin{pmatrix} 3^{N^N}I_\ell \\ F_N\end{pmatrix}\in\M_{2\ell,\ell}(\R)$ where $F_N=(f_N^{(i,j)})_{(i,j)\in\{1,\ldots,\ell\}^2}$. Again in a similar fashion as before, one can show that there exists a constant $c>0$ depending only on $Y_1,\ldots,Y_\ell$ such that for all  $N\ge 1$, $\psi_\ell(A_\infty,B_N)\le c H(B_N)^{-N/\ell}$. Thus,
\begin{equation}\label{equation_B_N_finale_aborenosnvn}
\forall \kappa>0,\quad \forall N\ge \kappa\ell,\quad \psi_\ell(A_\infty,B_N)\le \frac{c}{H(B_N)^{\kappa}}.
\end{equation}
Notice that $\psi_\ell(A_\infty,B_N)$ tends to $0$ when $N$ tends to infinity. Therefore, there exist infinitely many pairwise distinct subspaces $B_N$ satisfying Inequality \eqref{equation_B_N_finale_aborenosnvn}, so for all $\kappa>0$, $\muexpA {n}{A_\infty}\ell\ell\ge \kappa$, therefore $\muexpA {n}{A_\infty}\ell\ell=+\infty$.

Let us finally consider the case $n>2\ell$. Let us denote by $\phi$ a rational isomorphism from $\R^{2\ell}$ to $\R^{2\ell}\times\{0\}^{n-2\ell}$. Let $A'=\phi(A)$; Theorem \ref{th_inclusion_sev_rationnel_apeivpinpiaenv} yields $A'\in\mathfrak I_n(\ell,\ell)_\ell$ and 
\[\muexpA n{A'}\ell\ell=\muexpA {2\ell}A\ell\ell=\beta\]
which allows us to extend the result to integers $n>2\ell$.
\end{proof}

\begin{remark}
Since $\ell\ge 1$ and since $\beta$ satisfies Inequality \eqref{condition_sur_beta_aomeibnoaidbnv}, one has $\alpha=\ell\beta\ge(3+\sqrt 5)/2$. In the case $\ell=1$, we fall back on a known result on the irrationality exponent of $\xi_{1,1}$: 
\[\mu\left(\sum_{k=0}^\infty \frac{e^{(1,1)}_k}{\theta^{\lfloor\alpha^k\rfloor}}\right)=\alpha\]
where $\mu(\cdot)$ stands for the irrationality exponent, $(e_k^{(1,1)})_{k\in\N}$ is a sequence with values in $\{2,3\}$, $\theta$ is a prime number strictly greater than $3$, and $\alpha$ is a real number greater than $(3+\sqrt 5)/2$. The arguments in Section 8 of \cite{levesley06} lead easily to this result, but the method developed here is different (and the case $\theta=3$ is not covered here). If $2\le\alpha<(3+\sqrt 5)/2$, one still has $\mu(\xi_{1,1})=\alpha$ thanks to Theorem 2 of \cite{bugeaud08}. 
\end{remark}

\bibliographystyle{alpha}
\bibliography{biblio_article_2}

\begin{thebibliography}{Jos21b}

\bibitem[Bug08]{bugeaud08}
Y.~Bugeaud.
\newblock {Diophantine approximation and Cantor sets}.
\newblock {\em {Mathematische Annalen}}, 341:677--684, 2008.

\bibitem[dS20]{saxce20}
N.~de~Saxc{\'e}.
\newblock {\em {Groupes arithm{\'e}tiques et approximation diophantienne}}.
\newblock {M{\'e}moire en vue d'obtenir l'habilitation {\`a} diriger des
  recherches}, {Universit{\'e} Sorbonne Paris Nord},
  {https://www.math.univ-paris13.fr/$\sim$desaxce/}, 2020.

\bibitem[Jos21a]{joseph21}
E.~Joseph.
\newblock {\em {Approximation rationnelle de sous-espaces vectoriels}}.
\newblock PhD thesis, {Universit\'e Paris-Saclay, defended on May 19th, 2021},
  {arXiv:2101.07648}, 2021.

\bibitem[Jos21b]{joseph21bis}
E.~Joseph.
\newblock {On the exponent of approximation for subspaces of $\R^n$}.
\newblock {\em {arXiv:2106.04313, submitted}}, 2021.

\bibitem[LSV06]{levesley06}
J.~Levesley, C.~Salp, and S.~L. Velani.
\newblock {On a problem of K. Mahler: Diophantine approximation and Cantor
  sets}.
\newblock {\em {Mathematische Annalen}}, 338(1):97--118, 2006.

\bibitem[Mos20]{moshchevitin20}
N.~Moshchevitin.
\newblock {\"U}ber die {W}inkel zwischen {U}nterr{\"a}umen.
\newblock {\em {Colloquium Mathematicum}}, 2020.

\bibitem[Rot55]{roth55}
K.~F. Roth.
\newblock {Rational Approximations to Algebraic Numbers}.
\newblock {\em Mathematika}, 2, 1955.

\bibitem[Sch67]{schmidt67}
W.~M. Schmidt.
\newblock On {H}eights of {A}lgebraic {S}ubspaces and {D}iophantine
  {A}pproximations.
\newblock {\em {Annals of Mathematics}}, 85(3):430--472, 1967.

\bibitem[Sil00]{silvester00}
J.~Silvester.
\newblock Determinants of block matrices.
\newblock {\em The Mathematical Gazette}, 84(501):460--467, 2000.

\end{thebibliography}

\end{document}